\documentclass[11pt,reqno]{amsart}

\usepackage{amsmath,amssymb,amsfonts,amsthm,mathrsfs,mathtools}
\usepackage[dvipsnames]{xcolor}
\usepackage{tikz}
\usepackage{pgfplots}
\usepgfplotslibrary{groupplots}
\pgfplotsset{compat=1.16}
\usepackage{booktabs}
\usepackage{array}
\usepackage[margin=1in]{geometry}
\usepackage{enumitem}
\usepackage{float}

\definecolor{lava}{rgb}{0.81,0.06,0.13}
\definecolor{Cblue}{rgb}{0.50,0.85,0.85}
\definecolor{lime}{HTML}{A6CE39}
\definecolor{figA}{HTML}{2A78D6}
\definecolor{figB}{HTML}{EB6834}
\definecolor{figC}{HTML}{1BAF7A}
\definecolor{figGrid}{HTML}{BBBBBB}

\usepackage[colorlinks,linkcolor=lava,citecolor=blue,urlcolor=brown,hypertexnames=false]{hyperref}

\numberwithin{equation}{section}
\theoremstyle{plain}
\newtheorem{theorem}{Theorem}[section]
\newtheorem{proposition}[theorem]{Proposition}
\newtheorem{lemma}[theorem]{Lemma}
\newtheorem{corollary}[theorem]{Corollary}

\theoremstyle{definition}
\newtheorem{definition}[theorem]{Definition}
\newtheorem{example}[theorem]{Example}

\theoremstyle{remark}
\newtheorem{remark}[theorem]{Remark}

\newcommand{\R}{\mathbb{R}}

\newcommand{\sphere}{\mathbb{S}}
\newcommand{\cone}{\mathcal{C}}
\newcommand{\eps}{\varepsilon}
\newcommand{\diag}{\mathrm{diag}}
\DeclareMathOperator{\RV}{RV}
\DeclareMathOperator{\tr}{tr}

\DeclareRobustCommand{\orcidicon}{%
  \begin{tikzpicture}
    \draw[lime, fill=lime] (0,0)
      circle [radius=0.16]
      node[white] {{\fontfamily{qag}\selectfont \tiny ID}};
    \draw[white, fill=white] (-0.0625,0.095) circle [radius=0.007];
  \end{tikzpicture}\hspace{-2mm}}
\foreach \x in {A, B}{%
  \expandafter\xdef\csname orcid\x\endcsname{%
    \noexpand\href{https://orcid.org/\csname orcidauthor\x\endcsname}{\noexpand\orcidicon}}}

\title[Fujita exponents beyond integrable forcing]
{Fujita exponents beyond integrable forcing}

\author[V.~Kumar and M.~Majdoub]{ Vishvesh Kumar\orcidA{} and Mohamed Majdoub\orcidB{}} 
\thanks{Corresponding author: Vishvesh Kumar, Email: \texttt{vishvesh.mat@iitbhu.ac.in}.}

\address[V.~Kumar]{
Department of Mathematical Sciences, Indian Institute of Technology (BHU),
Varanasi, Uttar Pradesh, 221005, India.}
\email{\tt vishvesh.mat@iitbhu.ac.in}
\email{\tt vishveshmishra@gmail.com}

\address[M.~Majdoub]{Department of Mathematics, College of Science, Imam Abdulrahman Bin Faisal University, P.O.~Box 1982, Dammam, Saudi Arabia.\newline
Basic and Applied Scientific Research Center, Imam Abdulrahman Bin Faisal University, P.O.~Box 1982, 31441, Dammam, Saudi Arabia.}
\email{\tt mmajdoub@iau.edu.sa}
\email{\tt  med.majdoub@gmail.com}
\email{\tt mohamed.majdoub@fst.rnu.tn}


\begin{document}

\begin{abstract}
We study the Fujita phenomenon for the forced semilinear heat equation
$$\partial_tu-\Delta u=|u|^p+\mathbf{w}(x),
\qquad (t,x)\in(0,\infty)\times\mathbb{R}^N, \qquad u(0, \cdot)=u_0.$$
Instead of the classical integrability assumption on the forcing term, we impose quantitative growth conditions, in the sense of regular variation, on its spatial mass
$$F(R)=\int\limits_{|x|\le R}\mathbf{w}(x)\,dx.$$
Combining Karamata's theory of regularly varying functions with the Mitidieri--Pohozaev test function method, we identify the Fujita threshold $p_F(\gamma)=\frac{N-\gamma}{N-\gamma-2}$ in terms of the growth index $\gamma$ of $F$. We prove that no global weak solution exists in the subcritical range, and that the critical exponent is also excluded when the slowly varying factor of $F$ diverges. The argument yields a universal a~priori bound on the mass of $\mathbf{w}$ along any global solution. As a consequence, a one-sided $\limsup$ condition on $F$ already suffices, and we show by an explicit sparse construction that this condition is strictly weaker than both the $\liminf$ condition and regular variation.

The framework extends to anisotropic forcings through operator regular variation, to sign-changing forcings through a Tauberian criterion based on the Gaussian--Laplace transform, and to space--time dependent, Riesz-potential and fractional problems. In particular, we recover in a unified way the Bandle--Levine--Zhang exponent $N/(N-2)$ and the discontinuity of the critical exponent for forcings of the form $t^m\mathbf{w}(x)$.
\end{abstract}\maketitle

%\tableofcontents

%==================================================================
\section{Introduction}
\label{sec:intro}

We study the nonexistence of global nonnegative solutions to the forced
semilinear heat equation
\begin{equation}\label{main}
\begin{cases}
\partial_t u-\Delta u=|u|^p+\mathbf{w}(x), & (t,x)\in(0,\infty)\times\R^N,\\
u(0,x)=u_0(x), & x\in\R^N,
\end{cases}
\end{equation}
where \(N\ge 1\), \(p>1\), and \(\mathbf{w}\in L^1_{\mathrm{loc}}(\R^N)\).
We also treat space--time forcings \(\mathbf{w}=\mathbf{w}(x,t)\), as well as nonlocal
variants. Solutions are understood in the weak (distributional) sense of
Definition~\ref{def:weak} below, and we focus throughout on nonnegative
solutions. We do not address local well-posedness under the minimal
assumptions on the forcing term used here. As is customary in this context,
we call \emph{global nonexistence} (or, informally, blow-up) the
nonexistence of global-in-time weak solutions. Under stronger assumptions on
\(u_0\) and \(\mathbf{w}\), standard semigroup arguments provide local mild
solutions. Since every global mild solution is a global weak solution, our
nonexistence results may then be read as genuine finite-time blow-up
statements.

The unforced problem (\(\mathbf{w}\equiv0\)) was settled by Fujita in his
seminal paper \cite{Fujita1966}. Every nontrivial nonnegative solution
blows up in finite time if \(1<p<1+\tfrac{2}{N}\), while small data produce
global solutions if \(p>1+\tfrac{2}{N}\). The critical exponent
\(p=1+\frac2N\) belongs to the blow-up case, as shown by Hayakawa
\cite{Hayakawa1973} for \(N\le2\) and by Kobayashi, Sirao and Tanaka
\cite{KobayashiSiraoTanaka1977} in general; see also
\cite{Weissler1981,Qi1998}, the surveys \cite{Levine1990,DengLevine2000}, and
the monograph \cite{QuittnerSouplet2007}. For the fractional Laplacian
\((-\Delta)^s\), the Fujita exponent becomes \(1+\frac{2s}{N}\)
\cite{Sugitani1975,GueddaKirane1999,KiraneLaskriTatar2005,FinoKirane}.

A forcing term changes the picture drastically. Zhang
\cite{Zhang1998JMAA,Zhang1998} discovered that, for a nonnegative
nontrivial forcing, the critical exponent jumps from the Fujita value to
\(\frac{N}{N-2}\). Bandle, Levine and Zhang \cite{BandleLevineZhang2000}
then showed that, for \(N\ge3\) and an integrable forcing with
\(\int_{\R^N}\mathbf{w}(x)\,dx>0\), problem \eqref{main} has no global
solution when \(1<p<\frac{N}{N-2}\), and that the borderline exponent
\(p=\frac{N}{N-2}\) is also excluded under mild additional assumptions on
\(\mathbf{w}\). Conversely, global solutions exist for small data when
\(p>\frac{N}{N-2}\). In dimensions \(N=1,2\), positive mass prevents global
existence for every \(p>1\). We also refer to \cite{Pinsky1999} for related
results, and to \cite{JleliSamet2026zero} for the zero-mass case left open
in \cite{BandleLevineZhang2000}.

Forcings depending on time were introduced by Jleli, Kawakami and Samet
\cite{JleliKawakamiSamet2020}, who considered \(t^\sigma\mathbf{w}(x)\) with
\(\sigma>-1\) and \(\mathbf{w}\in L^1(\R^N)\). They proved that the critical
exponent equals \(\frac{N-2\sigma}{N-2\sigma-2}\) for \(-1<\sigma<0\), whereas
no global solution exists for any \(p>1\) when \(\sigma>0\). The critical
exponent is therefore discontinuous at \(\sigma=0\). The second author
showed in \cite{Majdoub2021} that, for general time profiles, nonexistence
is governed by the behaviour of the profile at infinity. This line of
research was continued for Hardy--H\'enon nonlinearities
\cite{Majdoub2023,AlshehriEtAl2023}, systems \cite{FinoJleliSamet2021},
nonlocal nonlinearities \cite{DCDS}, space--time fractional diffusion
\cite{MedJM}, mixed local--nonlocal diffusion
\cite{KumarTorebek,KumarTorebek2025criteria,arXiv}, degenerate and singular
equations \cite{MajdoubTorebek2026}, interfaces \cite{MajdoubMliki2026},
and the Heisenberg group \cite{BorikhanovRuzhanskyTorebek,FinoRuzhanskyTorebek2024}.
Discontinuous critical exponents also arise for combined nonlinearities
\cite{JleliSametSouplet2020}.

\begin{table}[H]
\centering
\small
\renewcommand{\arraystretch}{1.25}
\begin{tabular}{@{}>{\raggedright}p{4.1cm}>{\raggedright}p{4.4cm}>{\raggedright\arraybackslash}p{4.9cm}@{}}
\toprule
Problem & Assumption on the forcing & Critical exponent (blow-up below)\\
\midrule
\(\partial_tu-\Delta u=u^p\) \cite{Fujita1966,Hayakawa1973,KobayashiSiraoTanaka1977}
& none (\(\mathbf{w}\equiv0\))
& \(1+\frac2N\), critical case included\\
\(\partial_tu+(-\Delta)^su=u^p\) \cite{Sugitani1975,GueddaKirane1999,KiraneLaskriTatar2005}
& none
& \(1+\frac{2s}N\)\\
\(\partial_tu-\Delta u=|u|^p+\mathbf{w}(x)\) \cite{Zhang1998JMAA,BandleLevineZhang2000}
& \(\mathbf{w}\in L^1\), \(\int\mathbf{w}>0\), \(N\ge3\)
& \(\frac{N}{N-2}\), critical case included\\
\(\partial_tu-\Delta u=|u|^p+t^\sigma\mathbf{w}(x)\) \cite{JleliKawakamiSamet2020,Majdoub2021}
& \(\mathbf{w}\in L^1\), \(\int\mathbf{w}>0\)
& \(\frac{N-2\sigma}{N-2\sigma-2}\) if \(-1<\sigma<0\); \(\infty\) if \(\sigma>0\)\\
\(\partial_tu+(-\Delta)^su=|u|^p+\mathbf{w}(x)\) \cite{Majdoub2023,KumarTorebek}
& \(\mathbf{w}\in L^1\), \(\int\mathbf{w}>0\)
& \(\frac{N}{N-2s}\)\\
\midrule
\eqref{main}, \emph{this paper}
& \(F\in\RV_\gamma\), or \(\limsup R^{-\gamma}F(R)>0\)
& \(p_F(\gamma)=\frac{N-\gamma}{N-\gamma-2}\); critical case included if \(R^{-\gamma}F(R)\) is unbounded\\
\bottomrule
\end{tabular}
\caption{Some known Fujita-type thresholds and the threshold obtained in
the present paper. In the last row, \(F(R)=\int_{|x|\le R}\mathbf{w}\,dx\)
and \(\gamma\in[0,N)\); for \(\gamma\ge N-2\) we set \(p_F(\gamma)=\infty\).}
\label{tab:known}
\end{table}

All these works share one feature: the forcing is integrable in space, or
at least its total mass is well defined. Their arguments rely on the fact
that the mass of \(\mathbf{w}\) in a ball of radius \(R\) converges to a
positive limit as \(R\to\infty\). Such arguments do not apply to forcings
that are non-integrable, anisotropic, or growing at infinity. For the
biharmonic heat equation in exterior domains, Tobakhanov and Torebek
\cite{TobakhanovTorebek} recently showed that pointwise lower bounds of the
form \(\mathbf{w}(x)\ge C|x|^{-\omega}\) with \(\omega<N\) shift the
threshold. Our goal is to develop a systematic theory of this effect for the
heat equation, based on the asymptotic behaviour of the mass function.

\subsection{Mass growth and the threshold \(p_F(\gamma)\)}
Motivated by the works quoted above, and in particular by
\cite{KumarTorebek,MedJM,DCDS,arXiv,Majdoub2021,Majdoub2023}, we replace the
integrability hypothesis by a \emph{quantitative} lower bound on the spatial
mass of \(\mathbf{w}\):
\begin{equation}\label{eq:massbound}
F(R):=\int_{|x|\le R}\mathbf{w}(x)\,dx,\qquad
\liminf_{R\to\infty}\frac{F(R)}{R^\gamma}>0,
\qquad \gamma\in[0,N).
\end{equation}
Under this hypothesis we show (Theorem~\ref{thm:RVblowup}\,(i) below) that
\eqref{main} admits no global nonnegative solution for
\(1<p<p_F(\gamma)\), where
\begin{equation}\label{eq:pFgamma}
p_F(\gamma):=
\begin{cases}
\dfrac{N-\gamma}{N-\gamma-2}, & \gamma<N-2,\\[6pt]
\infty, & \gamma\ge N-2.
\end{cases}
\end{equation}
The case \(\gamma=0\) recovers the framework of
\cite{BandleLevineZhang2000}, whereas \(\gamma>0\) allows \(\mathbf{w}\) to
have infinite mass. The threshold \(p_F(\gamma)\) increases with
\(\gamma\) and becomes infinite at \(\gamma=N-2\); see
Figure~\ref{fig:pF}. Heuristically, \(p<p_F(\gamma)\) means
\(\gamma>N-2p'\), where \(p'=p/(p-1)\). A forcing behaving like
\(|x|^{-a}\) at infinity has \(\gamma=N-a\), so the condition reads
\(a<2p'\). The exponent \(2p'=\frac{2p}{p-1}\) is exactly the decay rate of
the singular stationary solutions of \(-\Delta v=v^p\), and this is the
scale at which the forcing competes with the nonlinearity.

\begin{figure}[H]
\centering
\begin{tikzpicture}
\begin{axis}[
  width=11cm, height=6.6cm,
  xmin=0, xmax=4.6, ymin=1, ymax=6,
  xlabel={growth index $\gamma$ of the mass $F$},
  ylabel={$p_F(\gamma)$},
  axis lines=left,
  grid=major, grid style={figGrid, very thin, opacity=0.5},
  tick label style={font=\small}, label style={font=\small},
  legend style={font=\small, draw=none, fill=white, fill opacity=0.9, text opacity=1, at={(axis cs:2.15,5.9)}, anchor=north west},
  legend cell align=left,
]
\addplot[figA, line width=1.2pt, domain=0:0.83, samples=200] {(3-x)/(1-x)};
\addlegendentry{$N=3$}
\addplot[figB, line width=1.2pt, dashed, domain=0:1.6, samples=200] {(4-x)/(2-x)};
\addlegendentry{$N=4$}
\addplot[figC, line width=1.2pt, dotted, domain=0:3.6, samples=200] {(6-x)/(4-x)};
\addlegendentry{$N=6$}
\addplot[black!60, thin, dashed] coordinates {(1,1) (1,6)};
\addplot[black!60, thin, dashed] coordinates {(2,1) (2,6)};
\addplot[black!60, thin, dashed] coordinates {(4,1) (4,6)};
\addplot[only marks, mark=*, mark size=1.8pt, black] coordinates {(0,3) (0,2) (0,1.5)};
\node[font=\scriptsize, anchor=west] at (axis cs:0.05,3.15) {$\frac{N}{N-2}$ (BLZ)};
\end{axis}
\end{tikzpicture}
\caption{The threshold \(p_F(\gamma)=\frac{N-\gamma}{N-\gamma-2}\) for
\(N=3,4,6\). At \(\gamma=0\) it equals the Bandle--Levine--Zhang exponent
\(\frac{N}{N-2}\) (dots), and it tends to \(+\infty\) as \(\gamma\uparrow N-2\)
(vertical dashed lines). For \(\gamma\ge N-2\), no global weak solution exists
for any \(p>1\).}
\label{fig:pF}
\end{figure}

\subsection{Regular variation}
Hypothesis \eqref{eq:massbound} is one-sided and qualitative. Once
\(F\) is positive and asymptotically of order \(R^\gamma\), it is natural
to ask whether \(F\) is \emph{regularly varying} of index \(\gamma\) in the
sense of Karamata \cite{Karamata}, that is, \(F(R)=R^\gamma L(R)\) with
\(L\) slowly varying. The one-dimensional and multivariate theories of
regular variation, developed by Karamata, Bingham, Goldie and Teugels,
de~Haan, Resnick, Sharpe, Meerschaert, and Scheffler
\cite{BGT,Seneta,deHaanFerreira,Resnick1987,Resnick2002,Resnick2007,Sharpe1969,Meerschaert1988,MS},
are well suited to this question. Regular variation has already proved
efficient in elliptic problems \cite{CirsteaRadulescu2006,Maric2000} and,
more recently, in Fujita-type problems with time weights \cite{arXiv}. In
the present setting it allows us to:
\begin{itemize}[leftmargin=*]
\item replace \eqref{eq:massbound} by the regularly varying mass condition
\(F\in\RV_\gamma\), which yields nonexistence throughout the subcritical
range without any lower bound on the slowly varying factor, and also at the
critical exponent \(p=p_F(\gamma)\) under slow amplification;
\item derive a Tauberian sufficient condition for sign-changing forcings in
terms of a Gaussian--Laplace transform of \(\mathbf{w}\);
\item formulate an anisotropic Fujita-type criterion associated with an
exponent matrix \(E\), which covers forcings with prescribed scaling rates
in different directions.
\end{itemize}

\subsection{Main results}
The main results of the paper are summarised below and in
Table~\ref{tab:main}.

\begin{enumerate}[label=\textnormal{(\arabic*)},leftmargin=*]
\item \textbf{Master theorem (Theorem~\ref{thm:RVblowup}).}
If either the \(\liminf\) hypothesis \eqref{eq:massbound} holds or
\(F\in\RV_\gamma\), then \eqref{main} admits no global nonnegative solution
for \(1<p<p_F(\gamma)\). If, in addition, \(F\in\RV_\gamma\) with slowly
varying factor \(L(R)\to\infty\), then nonexistence also holds at the
critical exponent \(p=p_F(\gamma)\).

\item \textbf{A universal mass bound and the \(\limsup\) criterion
(Corollary~\ref{cor:limsup} and Proposition~\ref{prop:limsupweak}).}
Every global weak solution forces \(F(R)\le K\,R^{N-2p'}\) for all
\(R>0\), with \(K\) depending only on \(N\) and \(p\). Consequently, the
\(\liminf\) in \eqref{eq:massbound} may be replaced by a \(\limsup\). For
every \(\gamma\in(0,N)\), we construct a nonnegative forcing, concentrated on
annuli at super-exponentially separated scales, for which \(R^{-\gamma}F(R)\)
has positive \(\limsup\) and vanishing \(\liminf\). This forcing is not
covered by the \(\liminf\) hypothesis nor by regular variation, but it is
covered by the \(\limsup\) criterion.

\item \textbf{Operator regular variation (Theorem~\ref{thm:operator}).}
Consider forcings whose operator mass \(F^E\), computed on anisotropic boxes
associated with a symmetric positive definite exponent matrix \(E\), is
regularly varying of index \(|E|-q\) with \(q\in(0,|E|)\). Then no global
solution exists for
\[
1<p<p_F^E:=\frac{q}{q-2\lambda_{\min}}
\quad\text{if } q>2\lambda_{\min},
\qquad\text{and for every } p>1 \text{ if } q\le 2\lambda_{\min}.
\]
Here \(|E|=\tr(E)\) and \(\lambda_{\min}\) is the smallest eigenvalue of
\(E\). For \(E=I\), the threshold coincides with the isotropic exponent
\eqref{eq:pFgamma}. In general, the anisotropic and isotropic criteria are
not comparable. We exhibit a four-dimensional example in which, for a whole
interval of parameters, the anisotropic theorem gives a strictly larger
range than the isotropic one.

\item \textbf{Tauberian criterion for sign-changing forcings
(Theorem~\ref{thm:Tauberblow}).}
Assume that \(\mathbf{w}\) changes sign, that its negative part has mass of
strictly smaller order, and that its Gaussian--Laplace transform
\(\Lambda_{\mathbf{w}}(s)\) is regularly varying of index \(-\gamma/2\) as
\(s\to0^+\). Then \(F\in\RV_\gamma\) and the master theorem applies. An
explicit example shows that the assumption on the negative part cannot be
dropped.

\item \textbf{Space--time forcings and the cumulative functional
(Theorems~\ref{thm:spacetime} and~\ref{thm:Phidich}).}
For forcings whose mass grows like \(t^mR^\gamma\), nonexistence holds for
\(1<p<p_F(\gamma,m)\) when \(-1<m\le0\), and for every \(p>1\) when \(m>0\).
This extends the discontinuity phenomenon of
\cite{JleliKawakamiSamet2020,Majdoub2021} to infinite-mass forcings. The
functional
\[
\Phi(T):=\int_0^T\!\!\int_{|x|\le R(T)} \mathbf{w}(x,t)\,dx\,dt
\]
is regularly varying, and its index \(\alpha\), compared with the
test-function exponent \(\beta(p)\), organises the blow-up statements:
\(\alpha>\beta(p)\) forces nonexistence, and so does \(\alpha=\beta(p)\)
together with \(L\to\infty\).

\item \textbf{Anisotropic, Riesz, and fractional extensions.}
The same scheme applies to forcings depending only on a strict subset of the
variables (Theorem~\ref{thm:anisotropic}), to Riesz-potential forcings
(Theorem~\ref{thm:riesz}), and to fractional diffusion
(Theorem~\ref{thm:frac}).
\end{enumerate}

\begin{table}[H]
\centering
\small
\renewcommand{\arraystretch}{1.25}
\begin{tabular}{@{}>{\raggedright}p{4.6cm}>{\raggedright}p{4.9cm}>{\raggedright\arraybackslash}p{3.9cm}@{}}
\toprule
Hypothesis on the forcing & Nonexistence of global solutions & Result\\
\midrule
\(\liminf R^{-\gamma}F(R)>0\) or \(F\in\RV_\gamma\)
& \(1<p<p_F(\gamma)\)
& Theorem~\ref{thm:RVblowup}\,(i)\\
\(F\in\RV_\gamma\), \(L(R)\to\infty\), \(\gamma<N-2\)
& \(p=p_F(\gamma)\)
& Theorem~\ref{thm:RVblowup}\,(ii)\\
\(\limsup R^{-\gamma}F(R)>0\) (resp.\ \(=\infty\))
& \(1<p<p_F(\gamma)\) (resp.\ \(p=p_F(\gamma)\))
& Corollary~\ref{cor:limsup}\\
\(F^E\in\RV_{|E|-q}\) (operator mass)
& \(1<p<\frac{q}{q-2\lambda_{\min}}\)
& Theorem~\ref{thm:operator}\\
sign-changing, \(\Lambda_{\mathbf{w}}\in\RV_{-\gamma/2}(0^+)\), \(F_-=O(R^{\gamma'})\)
& \(1<p<p_F(\gamma)\)
& Theorem~\ref{thm:Tauberblow}\\
\(\int_{|x|\le R}\mathbf{w}(x,t)\,dx\ge C_0t^mR^\gamma\)
& \(1<p<p_F(\gamma,m)\) if \(m\le0\); all \(p>1\) if \(m>0\)
& Theorem~\ref{thm:spacetime}\\
same, with \((-\Delta)^s\)
& \(1<p<p_F^{\,s}(\gamma,m)\) if \(m\le0\); all \(p>1\) if \(m>0\)
& Theorem~\ref{thm:frac}\\
\bottomrule
\end{tabular}
\caption{Overview of the main nonexistence results. Here
\(p_F(\gamma,m)=\frac{N-\gamma-2m}{N-\gamma-2m-2}\) and
\(p_F^{\,s}(\gamma,m)=\frac{N-\gamma-2sm}{N-\gamma-2sm-2s}\), with the value
\(+\infty\) when the denominator is nonpositive.}
\label{tab:main}
\end{table}

\subsection{Relation with previous work}
The forced Fujita theory of Bandle, Levine and Zhang
\cite{BandleLevineZhang2000,Zhang1998} appears in our setting as the
endpoint \(\gamma=0\). Indeed, if \(\mathbf{w}\in L^1(\R^N)\) and
\(\int_{\R^N}\mathbf{w}\,dx>0\), then \(F(R)\to \int_{\R^N}\mathbf{w}\,dx\).
Hence \(F\in\RV_0\), and Theorem~\ref{cor:sign} recovers the nonexistence
range
\[
1<p<\frac{N}{N-2}
\qquad (N\ge3)
\]
of \cite{BandleLevineZhang2000}. Similarly, for \(\gamma=0\) and
\(-1<m\le0\), Theorem~\ref{thm:spacetime} recovers the subcritical part of
the results of \cite{JleliKawakamiSamet2020}.

The present framework goes well beyond the integrable case. It applies to
forcings with infinite mass (Examples~\ref{ex:power}--\ref{ex:decayL}), and
it yields the continuous family of enlarged Fujita ranges
\(1<p<p_F(\gamma)\), together with the critical exponent under slow
amplification. The novelties relative to the existing literature are the
following:
(a) the continuous family \(p_F(\gamma)\) and the closed (critical-inclusive)
range under slowly varying amplification
(Theorem~\ref{thm:RVblowup}\,(ii) and Theorem~\ref{thm:classification}\,(3));
(b) the universal mass bound of Corollary~\ref{cor:limsup}, which shows that
a \(\limsup\) condition suffices, together with the sparse forcings of
Proposition~\ref{prop:limsupweak}, which show that this is a genuine
extension;
(c) the operator-anisotropic Fujita exponents and the example of strict
improvement (Theorem~\ref{thm:operator} and Example~\ref{ex:N4});
(d) the Tauberian criterion for sign-changing forcings
(Theorem~\ref{thm:Tauberblow}); and
(e) the unifying \(\Phi\)-formalism (Theorem~\ref{thm:Phidich}).

The technical engine throughout is the test function method of Mitidieri
and Pohozaev~\cite{MitidieriPohozaev2001}. The regular variation hypotheses
control the lower bounds on the forcing, while Karamata-type growth bounds
for slowly varying functions sharpen the resulting contradictions. In the
fractional case, the nonlocal test-function computations are made rigorous
by the C\'ordoba--C\'ordoba convexity inequality
\cite{CordobaCordoba,Ju2005}, following the framework of \cite{FinoKirane}.\\

The paper is organized as follows. Section~\ref{sec:RV} collects the prerequisites on regular variation,
including the operator framework. Section~\ref{sec:main} states and proves
the master theorem, establishes the universal mass bound, and discusses
examples. Sparse forcings are studied in Section~\ref{sec:sharp}.
Section~\ref{sec:class} classifies the asymptotic regimes of \(F(R)\) and
treats angular structure. Section~\ref{sec:ext} contains the space--time,
pointwise and sign-changing extensions. The operator-anisotropic theory is
developed in Section~\ref{sec:operator}. Nonlocal and fractional extensions,
together with the \(\Phi\)-formalism, occupy Section~\ref{sec:nonlocal}.
Section~\ref{sec:open} closes the paper with open problems.

%==================================================================
\section{Regular variation: prerequisites}\label{sec:RV}

We summarise the elements of the theory of regular variation used in the
sequel, and we refer to \cite{BGT,Seneta,MS} and the references therein for
detailed expositions.

\subsection{Classical regular variation}

A measurable function $V:[A,\infty)\to(0,\infty)$, $A>0$, belongs to
$\RV_\rho$, the class of \emph{regularly varying functions of index
$\rho\in\R$}, if
\[
\lim_{t\to\infty}\frac{V(tx)}{V(t)}=x^\rho\qquad \text{for every }x>0.
\]
The case $\rho=0$ defines the class of \emph{slowly varying} functions.

\begin{theorem}[{Karamata representation, \cite[Th.~1.3.1]{BGT}}]\label{thm:karamata}
A measurable function $V:[A,\infty)\to(0,\infty)$ belongs to $\RV_\rho$ if
and only if there exists $a\ge A$ such that
\[
V(t)=t^\rho L(t),\qquad
L(t)=c(t)\exp\!\Bigl(\int_a^t \tfrac{\eps(s)}{s}\,ds\Bigr),
\qquad t\ge a,
\]
where $c$ and $\eps$ are measurable, $c(t)\to c_0\in(0,\infty)$ and
$\eps(t)\to 0$ as $t\to\infty$. In particular, $L$ is slowly varying.
\end{theorem}

\begin{theorem}[{Uniform convergence and Potter bounds, \cite[Th.~1.2.1 and Th.~1.5.6]{BGT}}]\label{thm:potter}
Let $V\in\RV_\rho$. The convergence $V(tx)/V(t)\to x^\rho$ is uniform on
every compact subset of $(0,\infty)$. Moreover, for every $\delta>0$ there
exists $t_0>0$ such that, for all $t,tx\ge t_0$,
\[
(1-\delta)\min(x^{\rho-\delta},x^{\rho+\delta})
\le\frac{V(tx)}{V(t)}\le(1+\delta)\max(x^{\rho-\delta},x^{\rho+\delta}).
\]
\end{theorem}

The following elementary growth estimate for slowly varying functions will
be used repeatedly. It follows directly from the Karamata representation
theorem; see also \cite[Prop.~1.3.6]{BGT}.

\begin{lemma}[Growth bounds for slowly varying functions]\label{lem:SVbounds}
Let $L\in\RV_0$. Then for every $\delta>0$ there exists $R_\delta>0$ such
that
\[
R^{-\delta}\le L(R)\le R^{\delta}\qquad\text{for all }R\ge R_\delta.
\]
\end{lemma}

\begin{theorem}[{Karamata integration, \cite[Prop.~1.5.8]{BGT}}]\label{thm:karamataint}
Let $V\in\RV_\rho$ with $\rho>-1$, and assume that $V$ is locally integrable
on $[0,\infty)$. Then
\[
\int_0^R V(r)\,dr\sim \frac{R\,V(R)}{\rho+1}\qquad(R\to\infty).
\]
Equivalently, if $V(r)=r^\rho L(r)$, then
$\int_0^R V\sim R^{\rho+1}L(R)/(\rho+1)$.
\end{theorem}

The next result is Karamata's Tauberian theorem \cite[Th.~1.7.1]{BGT}.

\begin{theorem}\label{thm:tauber}
Let $U:[0,\infty)\to[0,\infty)$ be nondecreasing and right-continuous, with
$U(0)=0$, and assume that its Laplace--Stieltjes transform
$\widehat U(s)=\int_{[0,\infty)} e^{-st}\,dU(t)$ is finite for every $s>0$.
For $\rho\ge 0$ and $L\in\RV_0$, we have
\[
\widehat U(s)\sim s^{-\rho}L(1/s)\quad(s\to 0^+)
\ \Longleftrightarrow\
U(t)\sim \frac{t^\rho L(t)}{\Gamma(\rho+1)}\quad(t\to\infty).
\]
\end{theorem}

\subsection{Regular variation on cones}

\begin{definition}[Regular variation on a cone]
\label{def:RVcone}
Let $\cone\subset\R^N\setminus\{0\}$ be an open cone. A measurable function
$f:\R^N\to\R$ is said to be \emph{regularly varying at infinity on $\cone$
with index $\rho\in\R$} if there exist a function $V\in\RV_\rho$ and a
non-trivial measurable function $h:\cone\to\R$ such that
\[
\lim_{t\to\infty}\frac{f(tx)}{V(t)}=h(x),
\qquad x\in\cone,
\]
with convergence locally uniform on compact subsets of $\cone$.
The limit function $h$ is then $\rho$-homogeneous, that is,
\[
h(cx)=c^\rho h(x)
\qquad\text{for all }x\in\cone,\ c>0.
\]
\end{definition}

The homogeneity of $h$ is checked as follows. For $x\in\cone$ and $c>0$, we
have $cx\in\cone$, so $f(tcx)/V(t)\to h(cx)$. On the other hand,
\(\frac{f(tcx)}{V(t)}=\frac{f((tc)x)}{V(tc)}\cdot\frac{V(tc)}{V(t)}\to h(x)\,c^\rho\).
Comparing the two limits, we obtain $h(cx)=c^\rho h(x)$.

\begin{remark}[Polar form]\label{rem:polar}\rm
Since $h$ is positively $\rho$-homogeneous, it is determined by its
restriction $g:=h|_{\cone\cap\sphere^{N-1}}$, which is measurable, and we
have
\[
h(x)=|x|^\rho g\!\left(\frac{x}{|x|}\right),\qquad x\in\cone.
\]
The function $f$ is said to be \emph{asymptotically isotropic} if
$g$ is constant, and \emph{asymptotically anisotropic in amplitude}
otherwise. In either case, the radial scaling rate is the same in every
direction.
\end{remark}

\subsection{Operator regular variation}

Regular variation on a cone allows anisotropy in amplitude, but it retains
a common radial scaling exponent in every direction. To accommodate
genuinely different scaling rates in different directions, one is naturally
led to matrix exponents and to the theory of operator regular variation
\cite{Sharpe1969,Meerschaert1988,MS}.

\begin{definition}[Operator regular variation]\label{def:ORV}
Let $E\in\R^{N\times N}$ be a matrix whose eigenvalues have positive real
parts, and define
\[
t^E:=\exp\!\bigl((\log t)E\bigr),\qquad t>0.
\]
Let $\cone\subset\R^N\setminus\{0\}$ be an open cone that is invariant under
the family of dilations $\{t^E:t>0\}$. A measurable function
$f:\R^N\to\R$ is said to be \emph{operator-regularly varying on $\cone$
with index $\rho\in\R$} if there exist a function $V\in\RV_\rho$ and a
non-trivial measurable function $h:\cone\to\R$ such that
\[
\lim_{t\to\infty}\frac{f(t^E x)}{V(t)}
=
h(x),
\qquad x\in\cone,
\]
with convergence locally uniform on compact subsets of $\cone$.
\end{definition}

Since all eigenvalues of $E$ have positive real parts, $\{t^E\}_{t>0}$ is a
one-parameter group of expanding linear automorphisms. In particular,
\[
\|t^E x\|\to\infty \quad\text{as } t\to\infty,
\qquad
\|t^E x\|\to 0 \quad\text{as } t\to 0^+,
\]
for every $x\neq 0$.

\begin{proposition}[Anisotropic homogeneity of the limit function]
\label{prop:anisohom}
Let $f$ be operator-regularly varying on $\cone$ with index $\rho$.
Then the limit function $h$ is $E$-homogeneous of degree $\rho$, that is,
\[
h(s^E x)=s^\rho h(x),
\qquad s>0,\ x\in\cone.
\]
\end{proposition}

\begin{proof}
We first note that $(ts)^E=t^Es^E$ and that $s^Ex\in\cone$. Since $f$ is
operator-regularly varying on $\cone$ with index $\rho$, we infer that, as
$t\to\infty$,
\[
\frac{f((ts)^Ex)}{V(t)}=\frac{f(t^E(s^Ex))}{V(t)}\to h(s^Ex)
\quad\text{and}\quad
\frac{f((ts)^Ex)}{V(t)}=\frac{f((ts)^Ex)}{V(ts)}\,\frac{V(ts)}{V(t)}\to h(x)\,s^\rho.
\]
Comparing the two limits, we obtain the desired identity.
\end{proof}

For later use in Section~\ref{sec:operator}, we recall the operator polar
coordinates associated with a symmetric positive definite matrix $E$; see
\cite{MS} for the general theory.

\begin{remark}\label{rem:opolar}\rm
Assume in addition that $E$ is symmetric positive definite. For every
$x\neq0$, the function $t\longmapsto |t^{-E}x|^2$
is strictly decreasing on $(0,\infty)$, since
\[
\frac{d}{dt}|t^{-E}x|^2
=-\frac{2}{t}\langle E\,t^{-E}x,t^{-E}x\rangle<0.
\]
Moreover, $|t^{-E}x|\to\infty$ as $t\to0^+$ and $|t^{-E}x|\to0$ as
$t\to\infty$. Hence there exists a unique $t>0$ such that $|t^{-E}x|=1$,
which we denote by $\tau(x)$. The map $(t,\theta)\mapsto t^E\theta$ from
$(0,\infty)\times \sphere^{N-1}$ onto $\R^N\setminus\{0\}$ is a
homeomorphism, and writing
\[
x=\tau(x)^E\theta(x),
\qquad
\theta(x):=\tau(x)^{-E}x\in\sphere^{N-1},
\]
gives the associated operator polar decomposition. Since
$t\mapsto|t^{-E}x|$ is decreasing, the operator balls are
\[
B_R^E:=\{0\}\cup\{x\neq0:\tau(x)\le R\}
      =R^E\,\overline B(0,1).
\]
They are ellipsoids, and
\[
|B_R^E|
=\det(R^E)\,|\overline B(0,1)|
=R^{\operatorname{tr}(E)}\omega_N,
\]
where $\omega_N$ is the volume of the Euclidean unit ball. Thus
$|E|:=\operatorname{tr}(E)$ plays the role of a \emph{homogeneous dimension}
of $\R^N$ with respect to $E$.

By Proposition~\ref{prop:anisohom}, the limit function $h$ from
Definition~\ref{def:ORV} (with $\cone=\R^N\setminus\{0\}$) admits the
representation
\[
h(x)=\tau(x)^\rho\,g(\theta(x)),\qquad g:=h|_{\sphere^{N-1}}.
\]
For $E=I$ this reduces to the classical Euclidean polar decomposition.
For more general matrices $E$, for instance non-symmetric matrices with
spectrum in the open right half-plane, analogous constructions are
available; see \cite{MS}.
\end{remark}

%==================================================================
\section{The master theorem and consequences}\label{sec:main}

In this section we present the master theorem, which unifies the
nonexistence results for several classes of forcing terms. We first fix the
notion of solution.

\begin{definition}[Global weak solution]\label{def:weak}
Let $u_0\in L^1_{\mathrm{loc}}(\R^N)$, $u_0\ge 0$, and
$\mathbf{w}\in L^1_{\mathrm{loc}}(\R^N)$. A \emph{global weak
solution} of \eqref{main} is a function
$u\in L^p_{\mathrm{loc}}([0,\infty)\times\R^N)$, $u\ge 0$ a.e., such that
\begin{equation}\label{eq:weakform}
\int_0^T\!\!\!\int_{\R^N}\bigl(|u|^p+\mathbf{w}\bigr)\varphi\,dx\,dt
+\int_{\R^N}u_0\,\varphi(0, x)\,dx
=-\int_0^T\!\!\!\int_{\R^N}u\,\partial_t\varphi\,dx\,dt
-\int_0^T\!\!\!\int_{\R^N}u\,\Delta\varphi\,dx\,dt
\end{equation}
for every $T>0$ and every nonnegative
$\varphi\in C^{1,2}([0,T]\times\R^N)$ which is compactly supported in
space and satisfies $\varphi(T,\cdot)\equiv 0$.
\end{definition}

Every global classical (or mild) nonnegative solution of \eqref{main} is a
global weak solution; this follows from a standard integration by parts.

We next replace the qualitative assumption \eqref{eq:massbound} by a
regularly varying mass condition.

\begin{definition}\label{def:RVforcing}
A function $\mathbf{w}\in L^1_{\mathrm{loc}}(\mathbb R^N)$ is said to have
\emph{regularly varying mass of index $\gamma\in[0,N)$} if its mass function
\[
F(R):=\int_{|x|\le R}\mathbf{w}(x)\,dx
\]
is eventually positive and belongs to $\RV_\gamma$. Equivalently,
\[
F(R)=R^\gamma L(R),
\qquad L\in\RV_0 .
\]
\end{definition}

\begin{remark}\label{rem:hypcompare}\rm
The hypotheses \eqref{eq:massbound} and Definition~\ref{def:RVforcing}
are not comparable. On the one hand, \eqref{eq:massbound} allows oscillating
mass functions that are not regularly varying. For instance, if
$0<\gamma<N$ and
$\mathbf{w}(x)=|x|^{-(N-\gamma)}\bigl(2+\sin\log|x|\bigr)\mathbf 1_{\{|x|\ge1\}}$,
the substitution $r=e^y$ gives
\[
F(R)=|\sphere^{N-1}|\bigl(R^\gamma P(\log R)-P(0)\bigr),
\qquad
P(y):=\frac2\gamma+\frac{\gamma\sin y-\cos y}{1+\gamma^2},\qquad R\ge1.
\]
The function $P$ is continuous, $2\pi$-periodic, nonconstant and bounded
below by $\frac2\gamma-\frac{1}{\sqrt{1+\gamma^2}}>0$. We infer that
$\liminf_{R\to\infty}R^{-\gamma}F(R)>0$, whereas
$F(e^{\pi}R)/F(R)$ has no limit as $R\to\infty$, so that
$F\notin\RV_\gamma$. On the other hand, Definition~\ref{def:RVforcing}
allows $F(R)=R^\gamma L(R)$ with a slowly varying factor $L(R)\to0$, for
which \eqref{eq:massbound} fails. Theorem~\ref{thm:RVblowup}\,(i) below
covers both cases; in particular, it does not require any uniform positive
lower bound on $L$. By Lemma~\ref{lem:SVbounds}, for every $\delta>0$,
\[
R^{-\delta}\lesssim L(R)\lesssim R^\delta,
\qquad R\gg1.
\]
Thus a slowly varying factor may vanish at infinity, but only at a
sub-polynomial rate, and this control is sufficient for the subcritical
argument.
\end{remark}

\subsection{A two-scale test-function estimate}
Throughout the paper we fix two functions $f,g\in C^\infty([0,\infty))$ with
$0\le f,g\le 1$, $g$ nonincreasing, and
\[
f(s)=\begin{cases}
1, & \tfrac12\le s\le \tfrac34,\\[2pt]
0, & s\in [0,\tfrac14]\cup[\tfrac45,\infty),
\end{cases}
\qquad
g(s)=\begin{cases}
1, & s\in [0,1],\\[2pt]
0, & s\geq 2.
\end{cases}
\]
For $p>1$ we write $p':=p/(p-1)$. The next estimate is the common upper
bound behind all our results. It holds for forcings of any sign, and the
time scale $T$ and the space scale $R$ are independent.

\begin{lemma}[Two-scale upper bound]\label{lem:twoscale}
Let $p>1$, $\mathbf{w}\in L^1_{\mathrm{loc}}(\R^N)$, and let $u$ be a global
weak solution of \eqref{main}. For $T,R>0$, set
\[
\psi_{T,R}(t,x):=f(t/T)^{p'}\,g\bigl(|x|^2/R^2\bigr)^{2p'} .
\]
Then there exists a constant $K>0$, depending only on $N$, $p$, $f$ and $g$,
such that, for all $T,R>0$,
\begin{equation}\label{eq:twoscale}
\frac12\int_0^T\!\!\!\int_{\R^N}|u|^p\psi_{T,R}\,dx\,dt
+\int_0^T\!\!\!\int_{\R^N}\mathbf{w}\,\psi_{T,R}\,dx\,dt
\le K\bigl(T^{1-p'}R^N+T\,R^{N-2p'}\bigr).
\end{equation}
\end{lemma}

\begin{proof}
Write $\psi:=\psi_{T,R}$, $f_T(t):=f(t/T)^{p'}$ and
$g_R(x):=G(|x|^2/R^2)$ with $G:=g^{2p'}$, so that $\psi=f_Tg_R$. Since
$p'>1$ and $2p'>2$, the function $\psi$ belongs to
$C^{1,2}([0,T]\times\R^N)$. It is nonnegative, supported in space in the
ball $\{|x|\le\sqrt2R\}$, and $\psi(T,\cdot)\equiv0$ because $f(1)=0$. Hence
$\psi$ is admissible in \eqref{eq:weakform}. Moreover, $f(0)=0$ gives
$\psi(0,\cdot)\equiv0$, so that the initial-data term vanishes. Using the
elementary inequality $-a\le|a|$, we deduce from \eqref{eq:weakform} that
\begin{equation}\label{eq3.2}
\int_0^T\!\!\!\int_{\R^N}|u|^p\psi\,dx\,dt
+\int_0^T\!\!\!\int_{\R^N}\mathbf{w}\,\psi\,dx\,dt
\le \int_0^T\!\!\!\int_{\R^N}|u|\bigl(|\partial_t\psi|+|\Delta\psi|\bigr)\,dx\,dt .
\end{equation}

We now estimate the derivatives of $\psi$. Since
$\partial_t\psi=p'\,T^{-1}f(t/T)^{p'-1}f'(t/T)\,g_R$ and the exponents of $f$
recombine as $(p'-1)p'+p'(1-p')=0$, a direct computation gives, on the
set $\{\psi>0\}$,
\begin{equation}\label{eq:pointwise1}
|\partial_t\psi|^{p'}\,\psi^{1-p'}
=\Bigl(\frac{p'}{T}\Bigr)^{p'}\bigl|f'(t/T)\bigr|^{p'}\,g_R(x)
\le \frac{C}{T^{p'}}\,\mathbf 1_{\{|x|\le\sqrt2R\}}.
\end{equation}
For the spatial part, we observe that $G'=2p'g^{2p'-1}g'$ and
$G''=2p'(2p'-1)g^{2p'-2}(g')^2+2p'g^{2p'-1}g''$. Since $0\le g\le1$ and
$2p'-2=2p'(1-1/p')$, we obtain
\[
|G'|+|G''|\le C\,g^{2p'-2}=C\,G^{\,1-1/p'} .
\]
Moreover,
$\Delta g_R=\frac{4|x|^2}{R^4}G''\bigl(\tfrac{|x|^2}{R^2}\bigr)
+\frac{2N}{R^2}G'\bigl(\tfrac{|x|^2}{R^2}\bigr)$, and $|x|^2\le2R^2$ on
the support of $g_R$. It follows that
$|\Delta g_R|\le C R^{-2}g_R^{\,1-1/p'}$, and therefore, on
$\{\psi>0\}$,
\begin{equation}\label{eq:pointwise2}
|\Delta\psi|^{p'}\,\psi^{1-p'}
=f_T\,|\Delta g_R|^{p'}g_R^{1-p'}
\le \frac{C}{R^{2p'}}\,f_T(t)\,\mathbf 1_{\{|x|\le\sqrt2R\}} .
\end{equation}
Note that $\partial_t\psi$ and $\Delta\psi$ vanish identically on the set
$\{(t,x):\,\psi(t,x)=0\}$, because every term in their expansions contains a
strictly positive power of $f$ or of $g$.

By Young's inequality, for $a\in\{|\partial_t\psi|,|\Delta\psi|\}$ we have
\[
|u|\,a\le \tfrac14\,|u|^p\,\psi+C\,a^{p'}\,\psi^{1-p'}
\qquad\text{on }\{\psi>0\}.
\]
The quantity $\int_0^T\!\int_{\R^N}|u|^p\psi\,dx\,dt$ is finite, since
$u\in L^p_{\mathrm{loc}}$ and $\psi$ has compact support. We may therefore
absorb $\frac12\int_0^T\!\int_{\R^N}|u|^p\psi$ into the left-hand side of
\eqref{eq3.2}. Integrating \eqref{eq:pointwise1} and \eqref{eq:pointwise2}
over $[0,T]\times\{|x|\le\sqrt2R\}$, and using $\int_0^Tf_T\,dt\le T$, we
conclude that
\[
\frac12\int_0^T\!\!\!\int_{\R^N}|u|^p\psi\,dx\,dt
+\int_0^T\!\!\!\int_{\R^N}\mathbf{w}\,\psi\,dx\,dt
\le C\,\bigl(T^{-p'}\cdot T+R^{-2p'}\cdot T\bigr)\,(\sqrt2R)^N\omega_N,
\]
which is \eqref{eq:twoscale}.
\end{proof}

\subsection{The master theorem}
We are now ready to state the main result of this section.

\begin{theorem}\label{thm:RVblowup}
Let $N\ge 1$, $\gamma\in[0,N)$, and let
$\mathbf{w}\in L^1_{\mathrm{loc}}(\R^N)$ be nonnegative with $\mathbf{w}\not\equiv 0$.
\begin{enumerate}[label=\textnormal{(\roman*)},leftmargin=*]
\item \emph{(Subcritical regime.)} Assume that either
\begin{enumerate}[label=\textnormal{(\alph*)},leftmargin=*]
\item $\displaystyle\liminf_{R\to\infty}R^{-\gamma}F(R)>0$, or
\item $\mathbf{w}$ has regularly varying mass of index $\gamma$.
\end{enumerate}
Then \eqref{main} admits no global weak solution whenever
$1<p<p_F(\gamma)$, where
\begin{equation}\label{eq:pFgamma1}
p_F(\gamma):=
\begin{cases}
\dfrac{N-\gamma}{N-\gamma-2},& \gamma<N-2,\\[6pt]
\infty,& \gamma\ge N-2.
\end{cases}
\end{equation}
\item \emph{(Critical regime with amplification.)} Assume that $\mathbf{w}$ has
regularly varying mass of index $\gamma\in[0,N-2)$, so that
$p_F(\gamma)<\infty$ and in particular $N\ge 3$, with slowly varying factor
$L(R)\to\infty$ as $R\to\infty$. Then \eqref{main} admits no global weak
solution for $p=p_F(\gamma)$.
\end{enumerate}
\end{theorem}

\begin{proof}
Assume, for contradiction, that \eqref{main} admits a global weak solution
$u$ in the sense of Definition~\ref{def:weak}. For $T>0$, we use the test
function of Lemma~\ref{lem:twoscale} at the parabolic scale $R=\sqrt T$,
namely
\[
\psi_T(t,x):=\psi_{T,\sqrt T}(t,x)=f_T(t)\,g_T(x),\qquad
f_T(t):=f(t/T)^{p'},\qquad
g_T(x):=g\bigl(|x|^2/T\bigr)^{2p'} .
\]
With $R=\sqrt T$, both terms on the right-hand side of \eqref{eq:twoscale}
equal $T^{1+\frac N2-p'}$. Hence
\begin{equation}\label{eq:UB}
\frac12\int_0^T\!\!\!\int_{\R^N}|u|^p\psi_T\,dx\,dt
+\int_0^T\!\!\!\int_{\R^N}\mathbf{w}\,\psi_T\,dx\,dt
\le C\,T^{\,1+\frac{N}{2}-p'}
\qquad\text{for every }T>0,
\end{equation}
where $C=2K$ does not depend on $T$.

We now bound the forcing term from below. Since $f_T\equiv 1$ on
$[T/2,2T/3]$, $g_T\equiv 1$ on $\{|x|\le\sqrt T\}$, and $\mathbf{w}\ge 0$, we
have
\begin{equation} \label{eq3.4}
\int_0^T\!\!\!\int_{\R^N}\mathbf{w}\,\psi_T\,dx\,dt
\ge \int_{T/2}^{2T/3}\!dt\int_{|x|\le\sqrt T}\mathbf{w}(x)\,dx
=\frac{T}{6}\,F\bigl(\sqrt T\bigr).
\end{equation}
Combining \eqref{eq:UB} and \eqref{eq3.4}, we obtain
\begin{equation}\label{eq:master}
F\bigl(\sqrt T\bigr)\le C\,T^{\,\frac{N}{2}-p'}\qquad\text{for every } T>0,
\end{equation}
with a new constant $C$ depending only on $N$ and $p$. We set
$\alpha:=\frac{N-\gamma}{2}-p'$ and observe the elementary equivalences
\[
\alpha<0 \iff p'>\tfrac{N-\gamma}{2} \iff 1<p<p_F(\gamma),
\qquad
\alpha=0 \iff p=p_F(\gamma)\ \ (\gamma<N-2).
\]
Indeed, if $\gamma\ge N-2$, then $\frac{N-\gamma}2\le1<p'$ for every $p>1$.
If $\gamma<N-2$, then $a:=\frac{N-\gamma}{2}>1$, and $p'>a$ is equivalent
to $p<\frac{a}{a-1}=p_F(\gamma)$.

\smallskip
\noindent(i)\,(a) Under the $\liminf$ hypothesis, there are $c>0$ and
$T_0\ge1$ with $F(\sqrt T)\ge c\,T^{\gamma/2}$ for $T\ge T_0$. Then
\eqref{eq:master} yields $c\le C\,T^{\alpha}$ for $T\ge T_0$. Since
$\alpha<0$, the right-hand side tends to $0$ as $T\to\infty$, and we reach
a contradiction for large $T$.

\smallskip
\noindent(i)\,(b) Under the regular variation hypothesis, we have
$F(\sqrt T)=T^{\gamma/2}L(\sqrt T)$, so that \eqref{eq:master} becomes
\begin{equation}\label{eq:masterRV}
L\bigl(\sqrt T\bigr)\le C\,T^{\alpha}\qquad(T\ge T_1),
\end{equation}
where $T_1$ is such that $F$ is positive on $[\sqrt{T_1},\infty)$. We pick
$\delta\in(0,-2\alpha)$. By Lemma~\ref{lem:SVbounds}, $L(\sqrt T)\ge
T^{-\delta/2}$ for $T$ large. It follows that $T^{-\delta/2}\le C\,T^{\alpha}$
with $-\delta/2>\alpha$, which is impossible for large $T$.

\smallskip
\noindent(ii) At $p=p_F(\gamma)$ we have $\alpha=0$, and
\eqref{eq:masterRV} reads $L(\sqrt T)\le C$. This contradicts
$L(R)\to\infty$.

In all cases, the assumed global solution cannot exist, which completes the
proof.
\end{proof}

\subsection{A universal mass bound}
The constant in \eqref{eq:master} is independent of $T$, of $u_0$ and of
the solution itself. This simple observation has a useful consequence:
a contradiction along \emph{one} sequence of scales is enough.

\begin{corollary}[Universal mass bound and the $\limsup$ criterion]\label{cor:limsup}
Let $p>1$ and let $\mathbf{w}\in L^1_{\mathrm{loc}}(\R^N)$ be nonnegative.
If \eqref{main} admits a global weak solution for some initial datum
$u_0\ge0$, then
\begin{equation}\label{eq:apriori}
F(R)\le K_{N,p}\,R^{\,N-2p'}\qquad\text{for every } R>0,
\end{equation}
where $K_{N,p}$ depends only on $N$ and $p$ (through the fixed cut-off
functions $f,g$). Consequently, the following assertions hold.
\begin{enumerate}[label=\textnormal{(\alph*)},leftmargin=*]
\item If $\limsup_{R\to\infty}R^{2p'-N}F(R)=\infty$, then \eqref{main}
admits no global weak solution.
\item If $\limsup_{R\to\infty}R^{-\gamma}F(R)>0$ for some $\gamma\in[0,N)$,
then \eqref{main} admits no global weak solution for $1<p<p_F(\gamma)$.
\item If $\gamma\in[0,N-2)$ and $\limsup_{R\to\infty}R^{-\gamma}F(R)=\infty$,
then \eqref{main} admits no global weak solution for $p=p_F(\gamma)$.
\end{enumerate}
\end{corollary}

\begin{proof}
Inequality \eqref{eq:apriori} is \eqref{eq:master} with $T=R^2$, which
holds for every $T>0$. Assertion (a) follows immediately. For (b), let
$p<p_F(\gamma)$, so that $\alpha=\frac{N-\gamma}2-p'<0$. By assumption there
exist $c>0$ and a sequence $R_k\to\infty$ with $F(R_k)\ge cR_k^\gamma$.
We infer that $R_k^{2p'-N}F(R_k)\ge cR_k^{-2\alpha}\to\infty$, and (a)
applies. For (c), we have $\alpha=0$ at $p=p_F(\gamma)$, that is,
$2p'-N=-\gamma$, and again (a) applies.
\end{proof}

\begin{remark}\label{rem:lowdim}
\rm The nontriviality of $u$ (or of the initial datum $u_0$) is not needed
in the proof, since the contradiction is generated entirely by the forcing
term. Furthermore, if $\gamma\ge N-2$, then $p_F(\gamma)=\infty$, and
nonexistence holds for every $p>1$. In particular, for $N\le2$ and every
$\gamma\in[0,N)$, we have $\gamma\ge N-2$, and hence nonexistence holds for
all $p>1$.
\end{remark}

\begin{remark}[Comparison with the qualitative version]\label{rem:quantqual}
\rm
Under hypothesis (a) alone, the argument yields a contradiction only when
$\alpha<0$. In the critical case $\alpha=0$, inequality \eqref{eq:master}
merely says that $R^{-\gamma}F(R)$ is bounded, and the method is
inconclusive. Corollary~\ref{cor:limsup}\,(c) shows that what is needed at
the critical exponent is the \emph{unboundedness} of $R^{-\gamma}F(R)$.
Regular variation with $L(R)\to\infty$ is the most natural way to guarantee
it: the slowly varying factor $L$ survives the cancellation of powers at
$p=p_F(\gamma)$, and its divergence produces the contradiction. Note also
that, by Corollary~\ref{cor:limsup}, $\limsup_{R\to\infty}L(R)=\infty$ is
already sufficient in Theorem~\ref{thm:RVblowup}\,(ii).
\end{remark}

\begin{remark}[Interpretation of the mass bound]\label{rem:apriori}\rm
Inequality \eqref{eq:apriori} is a necessary condition for global
solvability which involves only the forcing, and not the initial datum. For
$N\ge3$ and $p>\frac{N}{N-2}$ we have $N-2p'>0$, and \eqref{eq:apriori}
states that the mass of $\mathbf{w}$ in large balls grows at most like
$R^{N-2p'}$. This is the mass of $|x|^{-2p'}$, which is precisely the decay
of the singular stationary solutions $c_p|x|^{-2/(p-1)}$ of $-\Delta v=v^p$
composed with the nonlinearity. Whether a converse holds, at least for small
forcings, is discussed in Section~\ref{sec:open}.
\end{remark}

\subsection{Examples}
We now present several particular forcing terms. The corresponding
nonexistence results follow immediately from Theorem~\ref{thm:RVblowup}.

\begin{example}[Power forcings]\label{ex:power}
Let $\gamma\in(0,N)$ and $\mathbf{w}(x)=|x|^{-(N-\gamma)}\mathbf 1_{\{|x|\ge 1\}}$.
Then $F(R)=|\sphere^{N-1}|(R^\gamma-1)/\gamma$ for $R\ge1$, so that
$L(R)\to|\sphere^{N-1}|/\gamma$, and Theorem~\ref{thm:RVblowup}\,(i) yields
nonexistence for $1<p<p_F(\gamma)$.
\end{example}

\begin{example}[Logarithmic amplification]\label{ex:log}
For $\gamma\in(0,N)$ and $a>0$, let
\[
\mathbf{w}(x)=|x|^{-(N-\gamma)}\bigl(\log(2+|x|)\bigr)^{a}\,\mathbf 1_{\{|x|\ge 1\}}.
\]
By Karamata's integration theorem (Theorem~\ref{thm:karamataint}), we have
$F(R)\sim \frac{|\sphere^{N-1}|}{\gamma}R^\gamma(\log R)^{a}$, so that
$L(R)\to\infty$. For $\gamma\in(0,N-2)$,
Theorem~\ref{thm:RVblowup}\,(ii) gives nonexistence at $p=p_F(\gamma)$,
which lies beyond the reach of the qualitative hypothesis (a).
\end{example}

\begin{example}[Iterated logarithms]\label{ex:loglog}
The forcing
$\mathbf{w}(x)=|x|^{-(N-\gamma)}\log\log(e^e+|x|)\,\mathbf 1_{\{|x|\ge 1\}}$
yields $F(R)\sim \frac{|\sphere^{N-1}|}{\gamma}\,R^\gamma\log\log R$. The
amplification is genuine, since $L\to\infty$, although extremely slowly.
Hence Theorem~\ref{thm:RVblowup}\,(ii) still applies for
$\gamma\in(0,N-2)$.
\end{example}

\begin{example}[A decaying slowly varying factor]\label{ex:decayL}
The forcing
\[
\mathbf{w}(x)=|x|^{-(N-\gamma)}\bigl(\log(2+|x|)\bigr)^{-1}\mathbf 1_{\{|x|\ge 1\}}
\]
produces $F(R)\sim \frac{|\sphere^{N-1}|}{\gamma}\,R^\gamma/\log R$. Thus
$L(R)\to 0$, and hypothesis (a) fails. Nevertheless,
Theorem~\ref{thm:RVblowup}\,(i)\,(b) yields nonexistence in the whole
subcritical range $1<p<p_F(\gamma)$. At the critical exponent
$p=p_F(\gamma)$, inequality \eqref{eq:masterRV} becomes
$c(\log\sqrt T)^{-1}\le C$, which is consistent. The test-function method is
therefore silent there (see also Section~\ref{sec:open}).
\end{example}

%==================================================================
\section{Sparse forcings: $\liminf$ versus $\limsup$}\label{sec:sharp}

This section examines the role of the lower growth condition
\eqref{eq:massbound}. By Corollary~\ref{cor:limsup}\,(b), the $\liminf$ in
\eqref{eq:massbound} can be replaced by a $\limsup$ without losing the
conclusion of Theorem~\ref{thm:RVblowup}\,(i). We now show that this is a
genuine improvement. We construct mass functions with positive $\limsup$
but vanishing $\liminf$; such mass functions are not regularly varying, so
that neither hypothesis of Theorem~\ref{thm:RVblowup}\,(i) applies to them.

\begin{proposition}\label{prop:limsupweak}
Let $\gamma\in(0,N)$. There exists a nonnegative
$\mathbf{w}\in L^1_{\mathrm{loc}}(\R^N)$ such that
\[
\limsup_{R\to\infty}\frac{F(R)}{R^{\gamma}}>0
\qquad\text{while}\qquad
\liminf_{R\to\infty}\frac{F(R)}{R^{\gamma}}=0 .
\]
In particular, $F\notin\RV_\gamma$, and neither hypothesis of
Theorem~\textnormal{\ref{thm:RVblowup}\,(i)} is satisfied. Nevertheless,
by Corollary~\textnormal{\ref{cor:limsup}\,(b)}, problem \eqref{main} with
this forcing admits no global weak solution for $1<p<p_F(\gamma)$.
\end{proposition}

\begin{proof}
Let $R_k:=2^{2^k}$, $k\ge 1$, and define the annuli
\[
A_k:=\{x\in\R^N: R_k\le |x|\le 2R_k\}.
\]
Since $2R_k=2^{2^k+1}<2^{2^{k+1}}=R_{k+1}$ for every $k \geq 1$, the family
$(A_k)_{k \geq 1}$ is pairwise disjoint. We define
\[
\mathbf{w}:=\sum_{k\ge 1} a_k\,\mathbf 1_{A_k},
\qquad a_k:=R_k^{\gamma-N}.
\]
Then \(\mathbf{w}\ge0\) and \(\mathbf{w}\in L^1_{\mathrm{loc}}(\R^N)\), since
every ball meets only finitely many annuli. For each \(k\ge1\), the mass of
\(\mathbf{w}\) on \(A_k\) is
\[
m_k:=\int_{A_k}\mathbf{w}\,dx
= a_k\,|A_k|
= R_k^{\gamma-N}\,\omega_N\bigl((2R_k)^N-R_k^N\bigr)
= c_N R_k^\gamma,
\qquad c_N:=\omega_N(2^N-1)>0.
\]

\emph{Positive $\limsup$.} For \(R=2R_k\), the ball \(B_{2R_k}\) contains
the whole annulus \(A_k\). Hence
\[
F(2R_k)\ge m_k=c_N R_k^\gamma
= c_N\,2^{-\gamma}(2R_k)^\gamma,
\]
and we obtain
\(\limsup_{R\to\infty}R^{-\gamma}F(R)\ge c_N\,2^{-\gamma}>0\).

\emph{Vanishing $\liminf$.} We observe that
\[
\frac{R_{k+1}}2=2^{2^{k+1}-1}\ge 2^{2^k+1}=2R_k
\qquad\text{for all }k\ge1,
\]
since \(2^{k+1}-1-(2^k+1)=2^k-2\ge0\). Hence the ball \(B_{R_{k+1}/2}\)
meets only the annuli \(A_1,\dots,A_k\), and therefore
\[
F(R_{k+1}/2)=\sum_{j=1}^k m_j
= c_N\sum_{j=1}^k R_j^\gamma.
\]
The exponents \(2^k-2^j\), \(1\le j\le k\), are distinct nonnegative
integers. Consequently,
\[
\sum_{j=1}^k R_j^\gamma
=R_k^\gamma\sum_{j=1}^k 2^{-\gamma(2^k-2^j)}
\le R_k^\gamma\sum_{m=0}^\infty 2^{-\gamma m}
= C_\gamma R_k^\gamma,
\qquad C_\gamma:=\frac{1}{1-2^{-\gamma}} .
\]
It follows that
\[
\frac{F(R_{k+1}/2)}{(R_{k+1}/2)^\gamma}
\le c_NC_\gamma\,2^\gamma\Bigl(\frac{R_k}{R_{k+1}}\Bigr)^\gamma
= c_NC_\gamma\,2^\gamma\,2^{-\gamma 2^k}\xrightarrow[k\to\infty]{}0,
\]
and thus \(\liminf_{R\to\infty}R^{-\gamma}F(R)=0\).

\emph{Failure of regular variation.} For $k\ge2$, the same reasoning gives
\[
F(R_k)=\sum_{j=1}^{k-1}m_j\le c_NC_\gamma R_{k-1}^\gamma,
\qquad
F(2R_k)\ge m_k=c_N R_k^\gamma.
\]
Hence
\[
\frac{F(2R_k)}{F(R_k)}
\ge C_\gamma^{-1}\Bigl(\frac{R_k}{R_{k-1}}\Bigr)^\gamma
= C_\gamma^{-1}\,2^{\gamma 2^{k-1}}
\xrightarrow[k\to\infty]{}\infty .
\]
This is incompatible with regular variation of index \(\gamma\), which
would imply \(F(2R)/F(R)\to 2^\gamma\) as \(R\to\infty\). Therefore
\(F\notin\RV_\gamma\). The last assertion follows from
Corollary~\ref{cor:limsup}\,(b).
\end{proof}

\begin{figure}[H]
\centering
\begin{tikzpicture}
\begin{axis}[
  width=12.5cm, height=5.6cm,
  xmin=2, xmax=66, ymin=0, ymax=0.7,
  xlabel={$\log_2 R$},
  ylabel={$R^{-\gamma}F(R)/c_N$},
  axis lines=left,
  xtick={2,4,8,16,32,64},
  ytick={0,0.25,0.5},
  grid=major, grid style={figGrid, very thin, opacity=0.5},
  tick label style={font=\small}, label style={font=\small},
]
\addplot[figA, line width=1.1pt] coordinates {
(2.0000,0.00000) (2.0155,0.00463) (2.0310,0.00931) (2.0465,0.01405) (2.0620,0.01884) (2.0775,0.02368) (2.0930,0.02858) (2.1085,0.03355)
(2.1240,0.03857) (2.1395,0.04366) (2.1550,0.04881) (2.1705,0.05403) (2.1860,0.05932) (2.2016,0.06468) (2.2171,0.07011) (2.2326,0.07561)
(2.2481,0.08120) (2.2636,0.08686) (2.2791,0.09261) (2.2946,0.09843) (2.3101,0.10435) (2.3256,0.11035) (2.3411,0.11644) (2.3566,0.12263)
(2.3721,0.12891) (2.3876,0.13529) (2.4031,0.14177) (2.4186,0.14835) (2.4341,0.15504) (2.4496,0.16183) (2.4651,0.16874) (2.4806,0.17576)
(2.4961,0.18290) (2.5116,0.19015) (2.5271,0.19753) (2.5426,0.20504) (2.5581,0.21267) (2.5736,0.22044) (2.5891,0.22834) (2.6047,0.23638)
(2.6202,0.24456) (2.6357,0.25288) (2.6512,0.26136) (2.6667,0.26998) (2.6822,0.27877) (2.6977,0.28771) (2.7132,0.29681) (2.7287,0.30609)
(2.7442,0.31553) (2.7597,0.32515) (2.7752,0.33495) (2.7907,0.34493) (2.8062,0.35510) (2.8217,0.36547) (2.8372,0.37603) (2.8527,0.38679)
(2.8682,0.39775) (2.8837,0.40893) (2.8992,0.42033) (2.9147,0.43194) (2.9302,0.44378) (2.9457,0.45585) (2.9612,0.46816) (2.9767,0.48071)
(2.9922,0.49351) (3.0000,0.50000) (3.0233,0.49200) (3.0388,0.48675) (3.0543,0.48154) (3.0698,0.47640) (3.0853,0.47130) (3.1008,0.46627)
(3.1163,0.46128) (3.1318,0.45635) (3.1473,0.45147) (3.1628,0.44665) (3.1783,0.44187) (3.1938,0.43715) (3.2093,0.43248) (3.2248,0.42785)
(3.2403,0.42328) (3.2558,0.41876) (3.2713,0.41428) (3.2868,0.40985) (3.3023,0.40547) (3.3178,0.40114) (3.3333,0.39685) (3.3488,0.39261)
(3.3643,0.38841) (3.3798,0.38426) (3.3953,0.38015) (3.4109,0.37609) (3.4264,0.37207) (3.4574,0.36416) (3.4884,0.35641) (3.5194,0.34884)
(3.5504,0.34142) (3.5814,0.33416) (3.6124,0.32705) (3.6434,0.32010) (3.6744,0.31329) (3.7054,0.30663) (3.7364,0.30011) (3.7674,0.29373)
(3.7984,0.28748) (3.8295,0.28137) (3.8605,0.27539) (3.8915,0.26953) (3.9225,0.26380) (3.9535,0.25819) (3.9845,0.25270) (4.0620,0.25832)
(4.0930,0.26297) (4.1240,0.26798) (4.1550,0.27334) (4.1860,0.27907) (4.2171,0.28519) (4.2481,0.29171) (4.2791,0.29864) (4.3101,0.30600)
(4.3411,0.31381) (4.3566,0.31788) (4.3721,0.32207) (4.3876,0.32639) (4.4031,0.33082) (4.4186,0.33539) (4.4341,0.34007) (4.4496,0.34489)
(4.4651,0.34984) (4.4806,0.35493) (4.4961,0.36015) (4.5116,0.36551) (4.5271,0.37102) (4.5426,0.37667) (4.5581,0.38246) (4.5736,0.38842)
(4.5891,0.39452) (4.6047,0.40078) (4.6202,0.40721) (4.6357,0.41379) (4.6512,0.42055) (4.6667,0.42747) (4.6822,0.43457) (4.6977,0.44185)
(4.7132,0.44931) (4.7287,0.45695) (4.7442,0.46478) (4.7597,0.47281) (4.7752,0.48103) (4.7907,0.48945) (4.8062,0.49807) (4.8217,0.50691)
(4.8372,0.51596) (4.8527,0.52522) (4.8682,0.53471) (4.8837,0.54442) (4.8992,0.55437) (4.9147,0.56455) (4.9302,0.57498) (4.9457,0.58564)
(4.9612,0.59656) (4.9767,0.60774) (4.9922,0.61918) (5.0000,0.62500) (5.0233,0.61501) (5.0388,0.60843) (5.0543,0.60193) (5.0698,0.59549)
(5.0853,0.58913) (5.1008,0.58283) (5.1163,0.57660) (5.1318,0.57044) (5.1473,0.56434) (5.1628,0.55831) (5.1783,0.55234) (5.1938,0.54644)
(5.2093,0.54060) (5.2248,0.53482) (5.2403,0.52910) (5.2558,0.52345) (5.2713,0.51785) (5.2868,0.51232) (5.3023,0.50684) (5.3178,0.50142)
(5.3333,0.49606) (5.3488,0.49076) (5.3643,0.48551) (5.3798,0.48033) (5.3953,0.47519) (5.4109,0.47011) (5.4264,0.46509) (5.4419,0.46012)
(5.4574,0.45520) (5.4729,0.45033) (5.4884,0.44552) (5.5039,0.44076) (5.5194,0.43604) (5.5349,0.43138) (5.5504,0.42677) (5.5659,0.42221)
(5.5814,0.41770) (5.5969,0.41323) (5.6124,0.40882) (5.6279,0.40445) (5.6434,0.40012) (5.6589,0.39585) (5.6744,0.39162) (5.6899,0.38743)
(5.7054,0.38329) (5.7209,0.37919) (5.7364,0.37514) (5.7519,0.37113) (5.7829,0.36324) (5.8140,0.35551) (5.8450,0.34795) (5.8760,0.34055)
(5.9070,0.33331) (5.9380,0.32623) (5.9690,0.31929) (6.0000,0.31250) (6.0310,0.30586) (6.0620,0.29935) (6.0930,0.29299) (6.1240,0.28676)
(6.1550,0.28066) (6.1860,0.27469) (6.2171,0.26885) (6.2481,0.26313) (6.2791,0.25754) (6.3101,0.25206) (6.3411,0.24670) (6.3721,0.24146)
(6.4031,0.23632) (6.4341,0.23130) (6.4651,0.22638) (6.4961,0.22157) (6.5271,0.21685) (6.5581,0.21224) (6.5891,0.20773) (6.6202,0.20331)
(6.6512,0.19899) (6.6822,0.19476) (6.7132,0.19062) (6.7442,0.18656) (6.7907,0.18065) (6.8372,0.17491) (6.8837,0.16936) (6.9302,0.16399)
(6.9767,0.15879) (7.0233,0.15375) (7.0698,0.14887) (7.1163,0.14415) (7.1628,0.13958) (7.2093,0.13515) (7.2558,0.13086) (7.3023,0.12671)
(7.3488,0.12269) (7.4109,0.11753) (7.4729,0.11258) (7.5349,0.10785) (7.5969,0.10331) (7.6589,0.09896) (7.7209,0.09480) (7.7984,0.08984)
(7.8760,0.08514) (7.9535,0.08068) (8.0310,0.08578) (8.0620,0.09367) (8.0775,0.09772) (8.0930,0.10183) (8.1085,0.10601) (8.1240,0.11026)
(8.1395,0.11458) (8.1550,0.11897) (8.1705,0.12344) (8.1860,0.12799) (8.2016,0.13261) (8.2171,0.13732) (8.2326,0.14211) (8.2481,0.14698)
(8.2636,0.15194) (8.2791,0.15699) (8.2946,0.16213) (8.3101,0.16736) (8.3256,0.17269) (8.3411,0.17812) (8.3566,0.18365) (8.3721,0.18927)
(8.3876,0.19501) (8.4031,0.20085) (8.4186,0.20680) (8.4341,0.21286) (8.4496,0.21904) (8.4651,0.22533) (8.4806,0.23175) (8.4961,0.23829)
(8.5116,0.24495) (8.5271,0.25175) (8.5426,0.25867) (8.5581,0.26573) (8.5736,0.27293) (8.5891,0.28027) (8.6047,0.28775) (8.6202,0.29538)
(8.6357,0.30317) (8.6512,0.31110) (8.6667,0.31920) (8.6822,0.32746) (8.6977,0.33588) (8.7132,0.34447) (8.7287,0.35323) (8.7442,0.36217)
(8.7597,0.37129) (8.7752,0.38060) (8.7907,0.39009) (8.8062,0.39978) (8.8217,0.40967) (8.8372,0.41975) (8.8527,0.43005) (8.8682,0.44055)
(8.8837,0.45127) (8.8992,0.46222) (8.9147,0.47338) (8.9302,0.48478) (8.9457,0.49641) (8.9612,0.50829) (8.9767,0.52041) (8.9922,0.53278)
(9.0000,0.53906) (9.0233,0.53044) (9.0388,0.52477) (9.0543,0.51916) (9.0698,0.51361) (9.0853,0.50812) (9.1008,0.50269) (9.1163,0.49732)
(9.1318,0.49200) (9.1473,0.48674) (9.1628,0.48154) (9.1783,0.47639) (9.1938,0.47130) (9.2093,0.46627) (9.2248,0.46128) (9.2403,0.45635)
(9.2558,0.45147) (9.2713,0.44665) (9.2868,0.44187) (9.3023,0.43715) (9.3178,0.43248) (9.3333,0.42785) (9.3488,0.42328) (9.3643,0.41876)
(9.3798,0.41428) (9.3953,0.40985) (9.4109,0.40547) (9.4264,0.40114) (9.4419,0.39685) (9.4574,0.39261) (9.4729,0.38841) (9.4884,0.38426)
(9.5039,0.38015) (9.5194,0.37609) (9.5349,0.37207) (9.5659,0.36416) (9.5969,0.35641) (9.6279,0.34884) (9.6589,0.34142) (9.6899,0.33416)
(9.7209,0.32705) (9.7519,0.32010) (9.7829,0.31329) (9.8140,0.30663) (9.8450,0.30011) (9.8760,0.29373) (9.9070,0.28748) (9.9380,0.28137)
(9.9690,0.27539) (10.0000,0.26953) (10.0310,0.26380) (10.0620,0.25819) (10.0930,0.25270) (10.1240,0.24733) (10.1550,0.24207) (10.1860,0.23692)
(10.2171,0.23188) (10.2481,0.22695) (10.2791,0.22213) (10.3101,0.21740) (10.3411,0.21278) (10.3721,0.20826) (10.4031,0.20383) (10.4341,0.19949)
(10.4651,0.19525) (10.4961,0.19110) (10.5271,0.18704) (10.5736,0.18110) (10.6202,0.17536) (10.6667,0.16979) (10.7132,0.16441) (10.7597,0.15919)
(10.8062,0.15414) (10.8527,0.14925) (10.8992,0.14452) (10.9457,0.13993) (10.9922,0.13549) (11.0388,0.13119) (11.0853,0.12703) (11.1318,0.12300)
(11.1938,0.11783) (11.2558,0.11287) (11.3178,0.10812) (11.3798,0.10357) (11.4419,0.09921) (11.5039,0.09504) (11.5814,0.09007) (11.6589,0.08535)
(11.7364,0.08089) (11.8140,0.07666) (11.8915,0.07265) (11.9845,0.06811) (12.0775,0.06386) (12.1860,0.05923) (12.2946,0.05494) (12.4186,0.05041)
(12.5426,0.04626) (12.6822,0.04199) (12.8372,0.03772) (13.0078,0.03351) (13.1938,0.02946) (13.4109,0.02534) (13.6589,0.02134) (13.9690,0.01721)
(14.3566,0.01316) (14.8837,0.00913) (15.7209,0.00511) (16.0310,0.01343) (16.0465,0.01812) (16.0620,0.02287) (16.0775,0.02767) (16.0930,0.03253)
(16.1085,0.03745) (16.1240,0.04243) (16.1395,0.04748) (16.1550,0.05259) (16.1705,0.05777) (16.1860,0.06302) (16.2016,0.06834) (16.2171,0.07373)
(16.2326,0.07920) (16.2481,0.08475) (16.2636,0.09037) (16.2791,0.09608) (16.2946,0.10187) (16.3101,0.10775) (16.3256,0.11371) (16.3411,0.11977)
(16.3566,0.12592) (16.3721,0.13216) (16.3876,0.13851) (16.4031,0.14495) (16.4186,0.15150) (16.4341,0.15815) (16.4496,0.16492) (16.4651,0.17179)
(16.4806,0.17878) (16.4961,0.18588) (16.5116,0.19311) (16.5271,0.20045) (16.5426,0.20793) (16.5581,0.21553) (16.5736,0.22327) (16.5891,0.23114)
(16.6047,0.23914) (16.6202,0.24730) (16.6357,0.25559) (16.6512,0.26404) (16.6667,0.27264) (16.6822,0.28139) (16.6977,0.29030) (16.7132,0.29938)
(16.7287,0.30863) (16.7442,0.31805) (16.7597,0.32764) (16.7752,0.33741) (16.7907,0.34737) (16.8062,0.35751) (16.8217,0.36785) (16.8372,0.37838)
(16.8527,0.38912) (16.8682,0.40006) (16.8837,0.41122) (16.8992,0.42258) (16.9147,0.43418) (16.9302,0.44599) (16.9457,0.45804) (16.9612,0.47032)
(16.9767,0.48285) (16.9922,0.49562) (17.0000,0.50211) (17.0233,0.49408) (17.0388,0.48880) (17.0543,0.48357) (17.0698,0.47840) (17.0853,0.47329)
(17.1008,0.46823) (17.1163,0.46322) (17.1318,0.45827) (17.1473,0.45337) (17.1628,0.44853) (17.1783,0.44373) (17.1938,0.43899) (17.2093,0.43430)
(17.2248,0.42966) (17.2403,0.42506) (17.2558,0.42052) (17.2713,0.41603) (17.2868,0.41158) (17.3023,0.40718) (17.3178,0.40283) (17.3333,0.39852)
(17.3488,0.39426) (17.3643,0.39005) (17.3798,0.38588) (17.3953,0.38175) (17.4109,0.37767) (17.4264,0.37364) (17.4574,0.36569) (17.4884,0.35792)
(17.5194,0.35030) (17.5504,0.34286) (17.5814,0.33557) (17.6124,0.32843) (17.6434,0.32145) (17.6744,0.31461) (17.7054,0.30792) (17.7364,0.30137)
(17.7674,0.29497) (17.7984,0.28869) (17.8295,0.28256) (17.8605,0.27655) (17.8915,0.27067) (17.9225,0.26491) (17.9535,0.25928) (17.9845,0.25377)
(18.0155,0.24837) (18.0465,0.24309) (18.0775,0.23792) (18.1085,0.23286) (18.1395,0.22791) (18.1705,0.22306) (18.2016,0.21832) (18.2326,0.21368)
(18.2636,0.20913) (18.2946,0.20469) (18.3256,0.20033) (18.3566,0.19607) (18.3876,0.19191) (18.4186,0.18782) (18.4651,0.18187) (18.5116,0.17610)
(18.5581,0.17051) (18.6047,0.16510) (18.6512,0.15986) (18.6977,0.15479) (18.7442,0.14988) (18.7907,0.14512) (18.8372,0.14052) (18.8837,0.13606)
(18.9302,0.13175) (18.9767,0.12757) (19.0233,0.12352) (19.0853,0.11832) (19.1473,0.11334) (19.2093,0.10857) (19.2713,0.10401) (19.3333,0.09963)
(19.3953,0.09544) (19.4574,0.09142) (19.5349,0.08664) (19.6124,0.08211) (19.6899,0.07781) (19.7674,0.07374) (19.8605,0.06914) (19.9535,0.06482)
(20.0465,0.06077) (20.1550,0.05637) (20.2636,0.05228) (20.3876,0.04798) (20.5271,0.04355) (20.6667,0.03954) (20.8217,0.03551) (21.0078,0.03121)
(21.2093,0.02714) (21.4419,0.02310) (21.7209,0.01904) (22.0620,0.01503) (22.5116,0.01101) (23.1783,0.00693) (24.4341,0.00290) (25.9380,0.00102)
(27.4419,0.00036) (28.9457,0.00013) (30.4496,0.00004) (31.9535,0.00002) (32.0155,0.00465) (32.0310,0.00933) (32.0465,0.01406) (32.0620,0.01885)
(32.0775,0.02370) (32.0930,0.02860) (32.1085,0.03356) (32.1240,0.03858) (32.1395,0.04367) (32.1550,0.04882) (32.1705,0.05404) (32.1860,0.05933)
(32.2016,0.06469) (32.2171,0.07012) (32.2326,0.07563) (32.2481,0.08121) (32.2636,0.08687) (32.2791,0.09262) (32.2946,0.09845) (32.3101,0.10436)
(32.3256,0.11036) (32.3411,0.11646) (32.3566,0.12264) (32.3721,0.12892) (32.3876,0.13530) (32.4031,0.14178) (32.4186,0.14836) (32.4341,0.15505)
(32.4496,0.16184) (32.4651,0.16875) (32.4806,0.17577) (32.4961,0.18291) (32.5116,0.19016) (32.5271,0.19754) (32.5426,0.20505) (32.5581,0.21268)
(32.5736,0.22045) (32.5891,0.22835) (32.6047,0.23639) (32.6202,0.24457) (32.6357,0.25289) (32.6512,0.26137) (32.6667,0.26999) (32.6822,0.27878)
(32.6977,0.28772) (32.7132,0.29682) (32.7287,0.30610) (32.7442,0.31554) (32.7597,0.32516) (32.7752,0.33496) (32.7907,0.34494) (32.8062,0.35511)
(32.8217,0.36547) (32.8372,0.37603) (32.8527,0.38680) (32.8682,0.39776) (32.8837,0.40894) (32.8992,0.42033) (32.9147,0.43195) (32.9302,0.44379)
(32.9457,0.45586) (32.9612,0.46817) (32.9767,0.48072) (32.9922,0.49351) (33.0000,0.50001) (33.0233,0.49201) (33.0388,0.48675) (33.0543,0.48155)
(33.0698,0.47640) (33.0853,0.47131) (33.1008,0.46627) (33.1163,0.46129) (33.1318,0.45636) (33.1473,0.45148) (33.1628,0.44665) (33.1783,0.44188)
(33.1938,0.43716) (33.2093,0.43248) (33.2248,0.42786) (33.2403,0.42329) (33.2558,0.41876) (33.2713,0.41429) (33.2868,0.40986) (33.3023,0.40548)
(33.3178,0.40114) (33.3333,0.39686) (33.3488,0.39261) (33.3643,0.38842) (33.3798,0.38427) (33.3953,0.38016) (33.4109,0.37610) (33.4264,0.37208)
(33.4574,0.36416) (33.4884,0.35642) (33.5194,0.34884) (33.5504,0.34142) (33.5814,0.33416) (33.6124,0.32706) (33.6434,0.32010) (33.6744,0.31330)
(33.7054,0.30664) (33.7364,0.30012) (33.7674,0.29373) (33.7984,0.28749) (33.8295,0.28137) (33.8605,0.27539) (33.8915,0.26954) (33.9225,0.26380)
(33.9535,0.25820) (33.9845,0.25270) (34.0155,0.24733) (34.0465,0.24207) (34.0775,0.23693) (34.1085,0.23189) (34.1395,0.22696) (34.1705,0.22213)
(34.2016,0.21741) (34.2326,0.21278) (34.2636,0.20826) (34.2946,0.20383) (34.3256,0.19950) (34.3566,0.19526) (34.3876,0.19110) (34.4186,0.18704)
(34.4651,0.18111) (34.5116,0.17536) (34.5581,0.16980) (34.6047,0.16441) (34.6512,0.15919) (34.6977,0.15414) (34.7442,0.14925) (34.7907,0.14452)
(34.8372,0.13993) (34.8837,0.13549) (34.9302,0.13120) (34.9767,0.12703) (35.0233,0.12300) (35.0853,0.11783) (35.1473,0.11287) (35.2093,0.10812)
(35.2713,0.10357) (35.3333,0.09921) (35.3953,0.09504) (35.4729,0.09007) (35.5504,0.08536) (35.6279,0.08089) (35.7054,0.07666) (35.7829,0.07265)
(35.8760,0.06811) (35.9690,0.06386) (36.0775,0.05923) (36.1860,0.05494) (36.3101,0.05041) (36.4341,0.04626) (36.5736,0.04200) (36.7287,0.03772)
(36.8992,0.03351) (37.0853,0.02946) (37.3023,0.02534) (37.5504,0.02134) (37.8605,0.01721) (38.2481,0.01316) (38.7752,0.00913) (39.6124,0.00511)
(41.1163,0.00180) (42.6202,0.00064) (44.1240,0.00022) (45.6279,0.00008) (47.1318,0.00003) (48.6357,0.00001) (50.1395,0.00000) (51.6434,0.00000)
(53.1473,0.00000) (54.6512,0.00000) (56.1550,0.00000) (57.6589,0.00000) (59.1628,0.00000) (60.6667,0.00000) (62.1705,0.00000) (63.6744,0.00000)
(65.0000,0.50000)
};
\addplot[black!60, thin, dashed, domain=2:66] {0.5};
\node[font=\scriptsize, anchor=south west] at (axis cs:40,0.5) {$2^{-\gamma}$};
\end{axis}
\end{tikzpicture}
\caption{The normalised mass $R^{-\gamma}F(R)/c_N$ of the sparse forcing of
Proposition~\ref{prop:limsupweak} for $N=3$ and $\gamma=1$, plotted against
$\log_2R$. The mass jumps to about $2^{-\gamma}$ on each annulus
$A_k=\{R_k\le|x|\le2R_k\}$, $\log_2R_k=2^k$, and then decays to almost $0$
before the next annulus. Hence the $\limsup$ is positive while the
$\liminf$ vanishes.}
\label{fig:sparse}
\end{figure}

\begin{remark}\label{rem:sparse}\rm
The above example concentrates the mass on annuli whose radii grow doubly
exponentially, which creates arbitrarily large gaps between successive
scales; see Figure~\ref{fig:sparse}. This allows \(R^{-\gamma}F(R)\) to
oscillate between values bounded away from zero and values arbitrarily close
to zero. It is the extreme sparseness of the support, rather than any subtle
cancellation, that drives the failure of regular variation. Since the upper
bound \eqref{eq:master} holds at \emph{every} scale with a uniform constant,
a single favourable sequence of scales is enough to produce a contradiction,
and the whole subcritical range $1<p<p_F(\gamma)$ is recovered. At the
critical exponent, however, the method is silent for this forcing. Indeed,
for \(R_k\le R\le R_{k+1}\) we have
\(F(R)\le c_NC_\gamma R_k^\gamma\le c_NC_\gamma R^\gamma\), so that
\(R^{-\gamma}F(R)\) is bounded and Corollary~\ref{cor:limsup}\,(c) does not
apply.
\end{remark}

%==================================================================
\section{Classification of regularly varying masses}\label{sec:class}

In this section we relate the asymptotics of the density $\mathbf{w}$ to
those of its mass $F$. Throughout, $\mathbf{w}\in L^1_{\mathrm{loc}}(\R^N)$ is
nonnegative, and we use the spherical average
\[
\overline{\mathbf{w}}(r):=\frac{1}{|\sphere^{N-1}|}
\int_{\sphere^{N-1}}\mathbf{w}(r\theta)\,d\sigma(\theta),
\qquad\text{so that}\qquad
F(R)=|\sphere^{N-1}|\int_0^R r^{N-1}\,\overline{\mathbf{w}}(r)\,dr .
\]

\begin{theorem}\label{thm:classification}
Let $\mathbf{w}\ge 0$ be locally integrable.
\begin{enumerate}[label=\textnormal{(\arabic*)},leftmargin=*]
\item \emph{(Integrable regime, $\gamma=0$.)}
\(F(R)\to M\in(0,\infty)\) as \(R\to\infty\) if and only if
\(\mathbf{w}\in L^1(\R^N)\) and \(\int_{\R^N}\mathbf{w}\,dx=M\). In that case,
$F\in\RV_0$ with $L(R)\to M$, and Theorem~\ref{thm:RVblowup}\,(i) gives
nonexistence for $1<p<p_F(0)=\frac{N}{N-2}$ $(N\ge3)$, which recovers the
framework of \cite{BandleLevineZhang2000}.
\item \emph{(Power regime, $0<\gamma<N$.)} If
$\overline{\mathbf{w}}(r)\sim c\,r^{-(N-\gamma)}$ as $r\to\infty$ for some $c>0$,
then
\[
F(R)\sim \frac{c\,|\sphere^{N-1}|}{\gamma}\,R^{\gamma},
\]
so that $F\in\RV_\gamma$ with $L(R)\to c|\sphere^{N-1}|/\gamma$.
\item \emph{(Borderline regime, $\gamma=0$ with amplification.)} If
$\overline{\mathbf{w}}(r)\sim c\,r^{-N}$ as $r\to\infty$ for some $c>0$, then
\[
F(R)\sim c\,|\sphere^{N-1}|\,\log R ,
\]
so that $F\in\RV_0$ with $L(R)\to\infty$. For $N\ge 3$,
Theorem~\ref{thm:RVblowup} then gives nonexistence in the closed range
$1<p\le p_F(0)=\frac{N}{N-2}$, critical exponent included.
\end{enumerate}
\end{theorem}

\begin{proof}
(1) Since \(\mathbf{w}\ge0\), the function $F$ is nondecreasing, and the
monotone convergence theorem gives
\[
\lim_{R\to\infty}F(R)=\int_{\R^N}\mathbf{w}(x)\,dx\in[0,\infty].
\]
Hence \(F(R)\to M\in(0,\infty)\) if and only if \(\mathbf{w}\in L^1(\R^N)\)
and \(\int_{\R^N}\mathbf{w}\,dx=M\). In this case $F$ is eventually positive and
\(F(\lambda R)/F(R)\to1\) for every $\lambda>0$, so that \(F\in\RV_0\) with
slowly varying factor \(L=F\to M\).

(2) Assume that $\overline{\mathbf{w}}(r)\sim c\,r^{-(N-\gamma)}$ as
$r\to\infty$, with \(0<\gamma<N\) and \(c>0\). Then
\[
V(r):=r^{N-1}\overline{\mathbf{w}}(r)\sim c\,r^{\gamma-1},
\]
hence \(V\) is eventually positive and \(V\in\RV_{\gamma-1}\). Since
\(\gamma-1>-1\) and \(V\) is locally integrable on $[0,\infty)$, Karamata's
integration theorem (Theorem~\ref{thm:karamataint}) yields
\[
\int_0^R V(r)\,dr \sim \frac{R\,V(R)}{\gamma}\sim\frac{c}{\gamma}R^\gamma .
\]
Multiplying by \(|\sphere^{N-1}|\), we obtain
\(F(R)\sim \frac{c\,|\sphere^{N-1}|}{\gamma}\,R^\gamma\). In particular
\(F\in\RV_\gamma\).

(3) Assume that $\overline{\mathbf{w}}(r)\sim c\,r^{-N}$ as $r\to\infty$, with
\(c>0\). Then \(r^{N-1}\overline{\mathbf{w}}(r)=r^{-1}(c+\eta(r))\) with
\(\eta(r)\to0\). We split
\[
F(R)=|\sphere^{N-1}|\int_0^1 r^{N-1}\overline{\mathbf{w}}(r)\,dr
    +|\sphere^{N-1}|\int_1^R \frac{c+\eta(r)}{r}\,dr .
\]
The first term is finite by local integrability, and the second one equals
$c\,|\sphere^{N-1}|\log R\,(1+o(1))$. We conclude that
\(F(R)\sim c\,|\sphere^{N-1}|\log R\), so that \(F\in\RV_0\) with slowly
varying factor \(L(R)\sim c|\sphere^{N-1}|\log R\to\infty\).
\end{proof}

\begin{remark}[A converse via monotone density]\label{rem:monodens}\label{rem:monodens2}\rm
The implications of Theorem~\ref{thm:classification}\,(2)--(3) cannot be
reversed in general: oscillatory densities may produce
\(F(R)\asymp R^\gamma\), or even \(F\in\RV_\gamma\), without
\(\overline{\mathbf{w}}\) admitting an exact power asymptotic. For
instance, the forcing of Remark~\ref{rem:hypcompare} gives \(F(R)\asymp R^\gamma\) with
\(F\notin\RV_\gamma\). To obtain \(F\in\RV_\gamma\), fix \(0<\gamma<N\) and let
\[
L(r):=2+\sin(\log\log r), \qquad r\ge e^e .
\]
Since \(\log\log(\lambda r)-\log\log r=\log\bigl(1+\frac{\log\lambda}{\log r}\bigr)\to0\)
and \(L\ge1\), the function \(L\) is slowly varying, but it does not
converge. The radial density
\[
\mathbf{w}(x):=|x|^{-(N-\gamma)}L(|x|)\mathbf 1_{\{|x|\ge e^e\}}+\mathbf 1_{\{|x|<e^e\}}
\]
satisfies \(\overline{\mathbf{w}}(r)=r^{-(N-\gamma)}L(r)\) for \(r\ge e^e\), so
that \(\overline{\mathbf{w}}\) has no exact power asymptotic. Nevertheless,
Karamata's theorem yields
\[
F(R)\sim \frac{|\sphere^{N-1}|}{\gamma}R^\gamma L(R),
\]
and hence \(F\in\RV_\gamma\).

A converse is recovered under monotonicity, up to a slowly varying factor.
Indeed, let
\[
U(r):=r^{N-1}\overline{\mathbf{w}}(r),
\qquad\text{so that}\qquad
F(R)=|\sphere^{N-1}|\int_0^R U(r)\,dr.
\]
If \(U\) is ultimately monotone and \(F\in\RV_\gamma\) for some
\(\gamma>0\), say \(F(R)=R^\gamma L(R)\), then the monotone density
theorem \cite[Th.~1.7.2]{BGT} yields
\[
U(r)\sim \frac{\gamma}{|\sphere^{N-1}|}\,\frac{F(r)}{r}
= \frac{\gamma}{|\sphere^{N-1}|}\,r^{\gamma-1}L(r),
\qquad\text{that is,}\qquad
\overline{\mathbf{w}}(r)\sim
\frac{\gamma}{|\sphere^{N-1}|}\,r^{-(N-\gamma)}L(r).
\]
In particular, if \(F(R)\sim C R^\gamma\) with \(C>0\), then
\(\overline{\mathbf{w}}(r)\sim \frac{\gamma C}{|\sphere^{N-1}|}\,r^{-(N-\gamma)}\),
which is the genuine converse of Theorem~\ref{thm:classification}\,(2).
\end{remark}

The next result treats densities with an angular profile.

\begin{theorem}[Spectral form]\label{thm:spectral}
Let $\mathbf{w}\ge 0$ be locally integrable, and suppose that there exist
$\rho\in[-N,0)$, a slowly varying function $\widetilde L$ which is locally
bounded on $[1,\infty)$, and a bounded measurable function
$g:\sphere^{N-1}\to[0,\infty)$ such that, for $r\ge1$ and a.e.\
$\theta\in\sphere^{N-1}$,
\begin{equation}\label{eq:profile}
\mathbf{w}(r\theta)=r^{\rho}\,\widetilde L(r)\,\bigl(g(\theta)+\eps(r,\theta)\bigr),
\qquad
\operatorname*{ess\,sup}_{\theta\in\sphere^{N-1}}|\eps(r,\theta)|\xrightarrow[r\to\infty]{}0 .
\end{equation}
Set $M_g:=\int_{\sphere^{N-1}}g\,d\sigma$.
\begin{enumerate}[label=\textnormal{(\roman*)},leftmargin=*]
\item If \(\rho\in(-N,0)\) and \(M_g>0\), then, with
\(\gamma:=N+\rho\in(0,N)\),
\[
F(R)\sim \frac{M_g}{\gamma}\,R^{\gamma}\,\widetilde L(R).
\]
In particular \(F\in\RV_\gamma\). Nonexistence holds for $1<p<p_F(\gamma)$,
and also at $p=p_F(\gamma)$ when $\gamma<N-2$ and $\widetilde L(R)\to\infty$.
\item If \(\rho=-N\) and \(M_g>0\), set
\(\widehat\ell(R):=\int_1^R \widetilde L(r)\,\frac{dr}{r}\). Then
\(\widehat\ell\) is slowly varying and
\[
F(R)=M_g\,\widehat\ell(R)\,\bigl(1+o(1)\bigr)+O(1)
\qquad(R\to\infty).
\]
If \(\widehat\ell(R)\to\infty\), then \(F(R)\sim M_g\,\widehat\ell(R)\), so
that \(F\in\RV_0\) with slowly varying factor tending to infinity; for
$N\ge 3$, nonexistence holds in the \emph{closed} range
$1<p\le \frac{N}{N-2}$. If \(\widehat\ell\) is bounded, then
\(\mathbf{w}\in L^1(\R^N)\) and Theorem~\ref{thm:classification}\,(1)
applies.
\item If \(M_g=0\), then \(g=0\) a.e.\ on \(\sphere^{N-1}\). In particular,
\[
\int_{\sphere^{N-1}} \mathbf{w}(r\theta)\,d\sigma
=o\bigl(r^\rho \widetilde L(r)\bigr).
\]
Hence, if \(\rho>-N\), then
\(F(R)=o\bigl(R^{N+\rho}\widetilde L(R)\bigr)+O(1)\), whereas if
\(\rho=-N\), then \(F(R)=o\bigl(\widehat\ell(R)\bigr)+O(1)\). In this
degenerate case, no nonexistence conclusion follows from \eqref{eq:profile}
alone.
\end{enumerate}
\end{theorem}

\begin{proof}
The uniformity in \eqref{eq:profile} allows integration over the compact
sphere. We obtain
\[
\int_{\sphere^{N-1}}\mathbf{w}(r\theta)\,d\sigma(\theta)
=r^\rho \widetilde L(r)\bigl(M_g+\eta(r)\bigr),
\qquad |\eta(r)|\le|\sphere^{N-1}|\operatorname*{ess\,sup}_\theta|\eps(r,\theta)|\to0 .
\]

(i) If \(\rho\in(-N,0)\), then, with \(\gamma=N+\rho\in(0,N)\),
\[
r^{N-1}\int_{\sphere^{N-1}}\mathbf{w}(r\theta)\,d\sigma(\theta)
=r^{\gamma-1}\widetilde L(r)\bigl(M_g+\eta(r)\bigr),
\]
which belongs to $\RV_{\gamma-1}$ since $M_g>0$. As \(\gamma-1>-1\),
Karamata's integration theorem (Theorem~\ref{thm:karamataint}) yields
\(F(R)\sim \frac{M_g}{\gamma}R^\gamma \widetilde L(R)\). The nonexistence
statements follow from Theorem~\ref{thm:RVblowup}.

(ii) If \(\rho=-N\), we write
\[
F(R)
=
C_0+M_g\,\widehat\ell(R)
+\int_1^R \widetilde L(r)\,\eta(r)\,\frac{dr}{r},
\qquad C_0:=\int_{|x|\le1}\mathbf{w}(x)\,dx .
\]
Given \(\eps>0\), we choose \(r_\eps\ge1\) with \(|\eta(r)|\le\eps\) for
\(r\ge r_\eps\). Since $\widetilde L$ is locally bounded, we deduce that
\[
\Bigl|\int_1^R \widetilde L(r)\,\eta(r)\,\frac{dr}{r}\Bigr|
\le C_\eps+\eps\,\widehat\ell(R),
\qquad R\ge1 .
\]
This gives \(F(R)=M_g\widehat\ell(R)(1+o(1))+O(1)\). The slow variation of
\(\widehat\ell\) follows from \cite[Prop.~1.5.9a]{BGT}. If
\(\widehat\ell(R)\to\infty\), the bounded terms are negligible and
\(F(R)\sim M_g\widehat\ell(R)\); the nonexistence statement then follows
from Theorem~\ref{thm:RVblowup}\,(i)--(ii) with $\gamma=0$. If
\(\widehat\ell\) is bounded, then $F$ is bounded, and since
$\mathbf{w}\ge0$ we infer that $\mathbf{w}\in L^1(\R^N)$.

(iii) If \(M_g=0\), then \(g=0\) a.e.\ since \(g\ge0\), and therefore
\(\int_{\sphere^{N-1}}\mathbf{w}(r\theta)\,d\sigma(\theta)
=r^\rho\widetilde L(r)\eta(r)=o\bigl(r^\rho \widetilde L(r)\bigr)\).
If \(\rho>-N\), Karamata's theorem applied as in (i), with $M_g+\eta$
replaced by $\eta$, gives
\(F(R)=o\bigl(R^{N+\rho}\widetilde L(R)\bigr)+O(1)\). If \(\rho=-N\), the
argument of (ii) gives \(F(R)=o\bigl(\widehat\ell(R)\bigr)+O(1)\). The final
claim is immediate.
\end{proof}

\begin{remark}\label{rem:profilesign}
\rm
For nonnegative densities $\mathbf{w}$, case (iii) is necessarily degenerate,
as stated. Angular cancellation ($M_g=0$ with $g\not\equiv0$) can only occur
for sign-changing densities. In that situation, the relevant criteria are
those of Theorem~\ref{cor:sign} and Theorem~\ref{thm:Tauberblow} below,
which are based on the \emph{net} mass.
\end{remark}

Finally, we record the strictly anisotropic case of densities depending
only on a strict subset of the coordinates. It may be viewed as a degenerate
limit of the operator framework of Section~\ref{sec:operator}.

\begin{theorem}\label{thm:anisotropic}
Let $1\le k\le N-1$, write $x=(y,z)\in\R^{k}\times\R^{N-k}$, and let
$\mathbf{w}(x)=\widetilde{\mathbf{w}}(y)$ with $0\le \widetilde{\mathbf{w}}\in L^1(\R^{k})$
and $\widetilde{\mathbf{w}}\not\equiv 0$. Then
\[
F(R)\sim \omega_{N-k}\,\|\widetilde {\mathbf{w}}\|_{L^1(\R^{k})}\,R^{\,N-k},
\]
so that $F\in\RV_{N-k}$. Consequently, \eqref{main} admits no global weak
solution for $1<p<p_F(N-k)$. In particular, if $k\le 2$, then
$N-k\ge N-2$ and nonexistence holds for \emph{every} $p>1$, while for
$k\ge3$ we have $p_F(N-k)=\frac{k}{k-2}$.
\end{theorem}

\begin{proof}
Writing \(x=(y,z)\in\R^k\times\R^{N-k}\) and using Fubini's theorem, we
have
\[
F(R)
=\int_{|y|\le R}\widetilde{\mathbf{w}}(y)
\left(\int_{|z|\le (R^2-|y|^2)^{1/2}}dz\right)dy
=\omega_{N-k}\int_{|y|\le R}\widetilde{\mathbf{w}}(y)\,(R^2-|y|^2)^{\frac{N-k}{2}}\,dy.
\]
Hence
\[
\frac{F(R)}{R^{N-k}}
=\omega_{N-k}\int_{\R^k}\widetilde{\mathbf{w}}(y)
\Bigl(1-\frac{|y|^2}{R^2}\Bigr)_+^{\frac{N-k}{2}}\,dy.
\]
The last factor lies in $[0,1]$ and converges pointwise to \(1\) as
\(R\to\infty\). Since $\widetilde{\mathbf{w}}\in L^1(\R^k)$, the dominated
convergence theorem gives
\[
\frac{F(R)}{R^{N-k}}
\to \omega_{N-k}\,\|\widetilde{\mathbf{w}}\|_{L^1(\R^k)}>0 .
\]
Therefore \(F\in\RV_{N-k}\), and the nonexistence claims follow from
Theorem~\ref{thm:RVblowup}. Finally, if $k\ge3$, then
$p_F(N-k)=\frac{k}{k-2}$.
\end{proof}

\begin{remark}\label{rem:anisoBLZ}\rm
For $k\ge3$, the threshold $\frac{k}{k-2}$ in Theorem~\ref{thm:anisotropic}
is the Bandle--Levine--Zhang exponent in dimension $k$. This is consistent
with the structure of the forcing: $\mathbf{w}$ is integrable in the $k$
variables on which it depends and constant in the remaining ones, so that
it behaves like an integrable forcing in $\R^k$.
\end{remark}

%==================================================================
\section{Space--time, pointwise and sign-changing extensions}\label{sec:ext}

In this section we discuss extensions of the preceding criteria. We first
allow space--time dependent forcings, then discuss pointwise lower bounds
implying the required mass growth, and finally treat forcings that may
change sign.

\subsection{Space--time forcings}
We consider the Cauchy problem
\begin{equation}\label{eq:mainST}
\partial_t u-\Delta u=|u|^p+\mathbf{w}(x,t),
\qquad (t,x)\in(0,\infty)\times\R^N,\qquad u(0,\cdot)=u_0,
\end{equation}
where $\mathbf{w}\in L^1_{\mathrm{loc}}([0,\infty)\times\R^N)$. Weak solutions
of \eqref{eq:mainST} are understood exactly as in Definition~\ref{def:weak},
with the space--time forcing $\mathbf{w}(x,t)$ in place of $\mathbf{w}(x)$. The
proof of Lemma~\ref{lem:twoscale} does not use the time independence of
$\mathbf{w}$, so that \eqref{eq:twoscale} holds verbatim for global weak
solutions of \eqref{eq:mainST}. For $\gamma\ge0$ and $m\in\R$ we set
\[
p_F(\gamma,m):=
\begin{cases}
\dfrac{N-\gamma-2m}{\,N-\gamma-2m-2\,}, & N-\gamma-2m>2,\\[6pt]
\infty, & N-\gamma-2m\le 2 .
\end{cases}
\]

\begin{theorem}\label{thm:spacetime}
Let $N\ge 1$, $m>-1$, $\gamma\ge 0$, and let $\mathbf{w}\ge 0$ satisfy
\begin{equation}\label{eq:STmass}
\int_{|x|\le R}\mathbf{w}(x,t)\,dx\ \ge\ C_0\,t^{m}R^{\gamma}
\qquad\text{for all } t\ge 1,\ R\ge R_0,
\end{equation}
for some $C_0>0$ and $R_0\ge1$.
\begin{enumerate}[label=\textnormal{(\roman*)},leftmargin=*]
\item Problem \eqref{eq:mainST} admits no global weak solution whenever
$1<p<p_F(\gamma,m)$.
\item If $m>0$, then \eqref{eq:mainST} admits no global weak solution for
every $p>1$.
\end{enumerate}
\end{theorem}

\begin{proof}
Assume that \eqref{eq:mainST} has a global weak solution, and let
$\psi_{T,R}$ be the test function of Lemma~\ref{lem:twoscale}. Since
$\mathbf{w}\ge0$ and $\psi_{T,R}\equiv1$ on $[T/2,2T/3]\times\{|x|\le R\}$,
we deduce from \eqref{eq:twoscale} and \eqref{eq:STmass} that, for
$T\ge2$ and $R\ge R_0$,
\[
C_0R^\gamma\int_{T/2}^{2T/3}t^m\,dt
\le\int_0^T\!\!\!\int_{\R^N}\mathbf{w}\,\psi_{T,R}\,dx\,dt
\le K\bigl(T^{1-p'}R^N+T\,R^{N-2p'}\bigr).
\]
Moreover, $t^m\ge c\,T^m$ on $[T/2,2T/3]$ with $c=c(m)>0$, so that the
left-hand side is bounded below by $c\,T^{m+1}R^\gamma$. (Any real $m$
would do at this point; the restriction $m>-1$ is the natural one for
forcings $t^m\mathbf{w}(x)$, which are then locally integrable up to $t=0$.)
Therefore
\begin{equation}\label{eq:STtwoscale}
T^{m+1}R^\gamma\le C\bigl(T^{1-p'}R^N+T\,R^{N-2p'}\bigr),
\qquad T\ge2,\ R\ge R_0 .
\end{equation}

(i) We choose the parabolic scale $R=\sqrt T$, which is admissible for
$T\ge R_0^2$. Then \eqref{eq:STtwoscale} becomes
$T^{m+1+\frac{\gamma}{2}}\le C\,T^{1+\frac N2-p'}$. The exponent on the left
is strictly larger precisely when $p'>\frac{N-\gamma-2m}{2}$, that is, when
$1<p<p_F(\gamma,m)$; indeed, if $N-\gamma-2m\le2$ this holds for every
$p>1$ since $p'>1$. In this range, the two bounds contradict each other for
large $T$.

(ii) We now freeze the spatial scale, $R=R_0$. Then \eqref{eq:STtwoscale}
reads $T^{m+1}\le C\bigl(T^{1-p'}+T\bigr)$ for $T\ge2$, with a constant
depending on $R_0$. Since $m>0$ and $1-p'<0<m+1$, this is impossible for
large $T$, whatever $p>1$.
\end{proof}

\begin{remark}\label{rem:STgamma}\rm
No upper bound on $\gamma$ is required. Spatially growing forcings, or
time-growing forcings with $m>0$, only enlarge the nonexistence range. Part
(ii) shows that, as soon as $m>0$, nonexistence holds for every $p>1$,
independently of $N$ and $\gamma$. For $\gamma=0$ this is the discontinuity
of the critical exponent at $m=0$ discovered in
\cite{JleliKawakamiSamet2020} (see also \cite{Majdoub2021}): for $-1<m<0$
the critical exponent is $p_F(0,m)=\frac{N-2m}{N-2m-2}$, whereas for $m>0$
it is infinite. Theorem~\ref{thm:spacetime} extends this dichotomy to
infinite-mass forcings. One can check that, for $-1<m\le0$, the parabolic
choice $R=\sqrt T$ is optimal among all scales $R=T^\theta$,
$\theta\ge0$, in \eqref{eq:STtwoscale}. A critical-case refinement
analogous to Theorem~\ref{thm:RVblowup}\,(ii), based on a uniformly
regularly varying version of \eqref{eq:STmass}, is contained in the
$\Phi$-framework of Section~\ref{sec:nonlocal}; see
Proposition~\ref{prop:PhiRV} and Theorem~\ref{thm:Phidich}.
\end{remark}

We now present a particular case, in which the forcing satisfies a
pointwise lower bound.

\begin{theorem}\label{thm:weighted}
Let $m>-1$ and $\beta<N$, and suppose that
\[
\mathbf{w}(x,t)\ \ge\ C\,t^{m}\,(1+|x|)^{-\beta}
\qquad\text{on } \R^N\times[1,\infty).
\]
Then \eqref{eq:mainST} admits no global weak solution for
$1<p<p_F(N-\beta,\,m)$, and for every $p>1$ if $m>0$.
\end{theorem}

\begin{proof}
For $R\ge1$ and $R/2\le|x|\le R$, we have
$(1+|x|)^{-\beta}\ge c_\beta R^{-\beta}$, with
$c_\beta=2^{-\beta}$ if $\beta\ge0$ and $c_\beta=2^{\beta}$ if $\beta<0$. Hence
\[
\int_{|x| \leq R} \mathbf{w}(x, t)\, dx
\geq Ct^m \int_{R/2 \leq |x|\le R}(1+|x|)^{-\beta}\,dx
\ge  c\,t^m R^{N-\beta}
\]
for $t\ge1$ and $R\ge1$. Consequently, \eqref{eq:STmass} holds with
$\gamma=N-\beta>0$ and $R_0=1$, and Theorem~\ref{thm:spacetime} gives the
result.
\end{proof}

\begin{remark}\label{rem:weightint}\rm
\leavevmode
\begin{enumerate}[label=\textnormal{(\roman*)},leftmargin=*]
\item For $\beta>N$, the weight is integrable in space and
\eqref{eq:STmass} holds with $\gamma=0$. The threshold becomes
$p_F(0,m)=\frac{N-2m}{N-2m-2}$ when $N-2m>2$, and $\infty$ otherwise. In the
borderline case $\beta=N$, the mass grows like $\log R$; this is the
amplified critical situation covered by Theorem~\ref{thm:Phidich}\,(ii).
\item Negative values of $\beta$ (forcings growing at infinity) are allowed
and give $\gamma=N-\beta>N$. Since $m>-1$, we have
$N-\gamma-2m=\beta-2m<2$, and nonexistence holds for every $p>1$.
\end{enumerate}
\end{remark}

\subsection{Sign-changing forcings: the net-mass criterion}

We now return to time-independent forcings and drop the sign condition on
$\mathbf{w}$. For $\mathbf{w}\in L^1_{\mathrm{loc}}(\R^N)$, possibly sign-changing,
the \emph{net mass}
\[
F(R)=\int_{|x|\le R}\mathbf{w}\,dx
\]
is a continuous function of $R$. This follows from dominated convergence,
since spheres are Lebesgue-null sets. The following result extends, in
particular, the subcritical part of the seminal theorem of Bandle, Levine and
Zhang \cite{BandleLevineZhang2000}.

\begin{theorem}\label{cor:sign}
Let $\mathbf{w}\in L^1_{\mathrm{loc}}(\R^N)$ be possibly sign-changing, and let
$\gamma\ge 0$.
\begin{enumerate}[label=\textnormal{(\roman*)},leftmargin=*]
\item If $\displaystyle\liminf_{R\to\infty}R^{-\gamma}F(R)>0$, or if $F$
is eventually positive with $F\in\RV_\gamma$, then \eqref{main} admits no
global weak solution for $1<p<p_F(\gamma)$.
\item If, moreover, $F\in\RV_\gamma$ with $\gamma\in[0,N-2)$ and slowly
varying factor $L(R)\to\infty$, then nonexistence of global weak solutions
holds also at $p=p_F(\gamma)$.
\end{enumerate}
In particular, taking $\gamma=0$ in (i), if $\mathbf{w}\in L^1(\R^N)$ and
$\int_{\R^N}\mathbf{w}\,dx>0$, then no global weak solution exists for
$1<p<\frac{N}{N-2}$ when $N\ge3$, and for every $p>1$ when $N\le2$.
\end{theorem}

\begin{proof}
Assume that \eqref{main} admits a global weak solution, and let
\(\psi_T=f_T(t)g_T(x)\) be the test function used in the proof of
Theorem~\ref{thm:RVblowup}. Lemma~\ref{lem:twoscale} does not require any
sign condition on the forcing. Discarding the nonnegative term
$\frac12\int\!\!\int|u|^p\psi_T$, we obtain
\begin{equation}\label{eq:UBsign}
\int_0^T\!\!\!\int_{\R^N}\mathbf{w}\,\psi_T\,dx\,dt
\le C\,T^{\,1+\frac N2-p'}.
\end{equation}
For the lower bound, we use the monotone radial structure of \(g_T\).
Recall that \(g_T(x)=G(|x|^2/T)\), where \(G:=g^{2p'}\) is continuous and
nonincreasing, with \(G\equiv1\) on \([0,1]\) and \(G\equiv0\) on
\([2,\infty)\). For \(\lambda\in(0,1)\) we set
\[
s(\lambda):=\sup\{s\ge0:\,G(s)>\lambda\}\in[1,2],
\qquad
\rho(\lambda,T):=\sqrt{T\,s(\lambda)} .
\]
Since $G$ is continuous and nonincreasing, we have
$\{s\ge0:G(s)>\lambda\}=[0,s(\lambda))$, and therefore
\(\{x:\,g_T(x)>\lambda\}=\{|x|<\rho(\lambda,T)\}\). As \(g_T\) is compactly
supported and \(0\le g_T\le1\), we have \(\mathbf{w}_\pm g_T\in L^1(\R^N)\).
Applying the layer-cake formula separately to \(\mathbf{w}_+\) and
\(\mathbf{w}_-\), and using that spheres are null sets, we obtain
\begin{equation}\label{eq:layercake}
\int_{\R^N}\mathbf{w}\,g_T\,dx
=\int_0^1\left(\int_{\{g_T>\lambda\}} \mathbf{w}\,dx\right)d\lambda
=\int_0^1 F\bigl(\rho(\lambda,T)\bigr)\,d\lambda .
\end{equation}

\smallskip
\noindent\emph{Case of the $\liminf$ hypothesis.} There are $c>0$ and
$R_0\ge1$ with $F(R)\ge c\,R^{\gamma}$ for $R\ge R_0$. Since
$\rho(\lambda,T)\ge\sqrt T$ for every $\lambda$, identity
\eqref{eq:layercake} yields, for $T\ge R_0^2$,
\[
\int_{\R^N}\mathbf{w}\,g_T\,dx\ \ge\ c\,T^{\gamma/2}\ (>0).
\]

\noindent\emph{Case $F\in\RV_\gamma$.} The ratios
$\rho(\lambda,T)/\sqrt T=s(\lambda)^{1/2}$ range over the compact set
$[1,\sqrt2]$. By the uniform convergence theorem
(Theorem~\ref{thm:potter}), we infer that
\[
\frac{F(\rho(\lambda,T))}{F(\sqrt T)}
\;\longrightarrow\; s(\lambda)^{\gamma/2}\ \ge\ 1
\qquad\text{uniformly in }\lambda\in(0,1),
\]
and hence, for all large $T$,
\[
\int_{\R^N}\mathbf{w}\,g_T\,dx
\ \ge\ \tfrac12\,F\bigl(\sqrt T\bigr)
=\tfrac12\,T^{\gamma/2}L\bigl(\sqrt T\bigr)\ (>0).
\]

In either case, the spatial integral is positive for large $T$. Since
\(\int_0^T f_T(t)\,dt\ge T/6\), we obtain
\[
\int_0^T\!\!\!\int_{\R^N}\mathbf{w}\,\psi_T\,dx\,dt
=\Bigl(\int_0^Tf_T\,dt\Bigr)\int_{\R^N}\mathbf{w}\,g_T\,dx
\ge \frac{T}{6}\int_{\R^N}\mathbf{w}\,g_T\,dx .
\]
Under the \(\liminf\) hypothesis, combining with \eqref{eq:UBsign} gives
\(c\,T^{1+\gamma/2}\le C\,T^{1+\frac N2-p'}\), that is,
\(c\le C\,T^{\frac{N-\gamma}{2}-p'}\). If \(p<p_F(\gamma)\), then
\(\frac{N-\gamma}{2}-p'<0\), and letting \(T\to\infty\) yields a
contradiction. Similarly, if \(F(R)=R^\gamma L(R)\in\RV_\gamma\), we obtain
\(c\,L(\sqrt T)\le C\,T^{\frac{N-\gamma}{2}-p'}\). This is impossible if
\(p<p_F(\gamma)\), by Lemma~\ref{lem:SVbounds}. When \(p=p_F(\gamma)\), it
is impossible provided \(L(R)\to\infty\). The conclusion follows.
\end{proof}

\begin{remark}\label{rem:signmech}\rm
The only structural ingredient beyond the nonnegative case is the
\emph{monotonicity} of the spatial cut-off. It allows us to rewrite the
smoothed mass \(\int_{\R^N}\mathbf{w}\,g_T\,dx\) as the layer-cake average
\eqref{eq:layercake} of the net mass \(F\) over radii comparable to
\(\sqrt T\). Thus the argument depends only on the lower behaviour of the
cumulative mass \(F\), and it requires no separate control of the negative
part \(\mathbf{w}_-\) beyond local integrability. In contrast with the
nonnegative case, a $\limsup$ condition along a sequence of radii is not
sufficient here, since $F$ need not be monotone; what is needed is a lower
bound for $F$ on whole intervals $[R_k,\sqrt2R_k]$ with $R_k\to\infty$.
\end{remark}

\subsection{A Tauberian criterion}

In applications, the net mass $F$ may be hard to estimate directly, while
the Gaussian-windowed mass
\[
\Lambda_{\mathbf{w}}(s):=\int_{\R^N}e^{-s|x|^2}\,\mathbf{w}(x)\,dx ,\qquad s>0,
\]
can often be computed by Fourier analysis or special-function methods. The
following criterion converts asymptotics of $\Lambda_{\mathbf{w}}$ into the
hypothesis of Theorem~\ref{cor:sign}. We write
$\mathbf{w}=\mathbf{w}_+-\mathbf{w}_-$ and
$F_{\pm}(R):=\int_{|x|\le R}{\mathbf{w}}_{\pm}\,dx$.

\begin{theorem}\label{thm:Tauberblow}
Let \(\mathbf{w}\in L^1_{\mathrm{loc}}(\R^N)\) be such that
\(\Lambda_{|\mathbf{w}|}(s)<\infty\) for every \(s>0\). Assume that:
\begin{enumerate}[label=\textnormal{(\alph*)},leftmargin=*]
\item there exist \(\gamma\in(0,N)\) and \(\ell\in\RV_0\) such that
\[
\Lambda_\mathbf{w}(s)\sim s^{-\gamma/2}\,\ell(1/s)
\qquad (s\to0^+);
\]
\item the negative part has mass of strictly smaller order, namely
\[
F_-(R)=O(R^{\gamma'})
\qquad\text{for some }\gamma'\in[0,\gamma)
\]
(in particular, \(\mathbf{w}_-\in L^1(\R^N)\) suffices, with \(\gamma'=0\)).
\end{enumerate}
Then
\begin{equation}\label{eq:taubconc}
F(R)\sim \frac{R^\gamma\,\ell(R^2)}
{\Gamma\!\left(1+\tfrac\gamma2\right)}
\qquad (R\to\infty).
\end{equation}
Hence \(F\in\RV_\gamma\) with slowly varying factor
\(L(R):=\frac{\ell(R^2)}{\Gamma(1+\gamma/2)}\). In particular,
Theorem~\textnormal{\ref{cor:sign}} applies: nonexistence holds
for \(1<p<p_F(\gamma)\), and also at \(p=p_F(\gamma)\) if
\(\gamma<N-2\) and \(\ell(R)\to\infty\).
\end{theorem}

\begin{proof}
We first show that the negative part is negligible for $\Lambda$. Splitting
at radius $s^{-1/2}$, we have
\[
\Lambda_{\mathbf{w}_-}(s)
\le F_-\bigl(s^{-1/2}\bigr)
+\int_{|x|>s^{-1/2}}e^{-s|x|^2}\mathbf{w}_-\,dx .
\]
For the tail, we write the integral as a Stieltjes integral in the radial
variable and integrate by parts. The boundary term at infinity vanishes
because $F_-$ grows polynomially. Using (b), we obtain, for $s$ small,
\begin{align*}
\int_{|x|>s^{-1/2}}e^{-s|x|^2}\mathbf{w}_-\,dx
&=\int_{s^{-1/2}}^{\infty}2sr\,e^{-sr^2}\bigl(F_-(r)-F_-(s^{-1/2})\bigr)\,dr\\
&\le C\int_{s^{-1/2}}^{\infty}2sr\,e^{-sr^2}r^{\gamma'}\,dr
= C' s^{-\gamma'/2},
\end{align*}
where the last identity follows from the substitution $r=v/\sqrt s$. Since
also $F_-(s^{-1/2})\le Cs^{-\gamma'/2}$, we deduce that
$\Lambda_{\mathbf{w}_-}(s)=O\bigl(s^{-\gamma'/2}\bigr)=o\bigl(s^{-\gamma/2}\ell(1/s)\bigr)$,
by Lemma~\ref{lem:SVbounds}. Hence (a) transfers to the positive part:
\[
\Lambda_{{\mathbf{w}}_+}(s)=\Lambda_{\mathbf{w}}(s)+\Lambda_{{\mathbf{w}}_-}(s)
\ \sim\ s^{-\gamma/2}\,\ell(1/s)\qquad(s\to0^+).
\]

We next apply Karamata's Tauberian theorem. The function
$U(u):=F_+\bigl(\sqrt u\bigr)$, $u\ge0$, is continuous, nondecreasing, and
$U(0)=0$. Writing the Gaussian window radially, we have
\[
\Lambda_{{\mathbf{w}}_+}(s)=\int_{\R^N}e^{-s|x|^2}{\mathbf{w}}_+(x)\,dx
=\int_0^\infty e^{-su}\,dU(u)=\widehat U(s),
\]
which is finite for all $s>0$ by hypothesis. Theorem~\ref{thm:tauber}
with index $\gamma/2$ gives
\[
U(u)\ \sim\ \frac{u^{\gamma/2}\,\ell(u)}{\Gamma(1+\tfrac\gamma2)}
\quad(u\to\infty),
\qquad\text{that is,}\qquad
F_+(R)\ \sim\ \frac{R^{\gamma}\,\ell(R^2)}{\Gamma(1+\tfrac\gamma2)} .
\]
Note that $R\mapsto\ell(R^2)$ is again slowly varying.

Finally, we return to the net mass. By (b), we have
$F(R)=F_+(R)-F_-(R)=F_+(R)+O(R^{\gamma'})$, and
$R^{\gamma'}=o\bigl(R^{\gamma}\ell(R^2)\bigr)$ again by
Lemma~\ref{lem:SVbounds}. This proves \eqref{eq:taubconc}. Consequently,
$F$ is eventually positive and $F\in\RV_\gamma$, so that
Theorem~\ref{cor:sign} applies.
\end{proof}

\begin{example}[A sign-changing forcing detected by its Laplace transform]
\label{ex:taubsign}
Let \(\gamma\in(0,N)\), \(A>0\), and
\[
\mathbf{w}(x)=|x|^{-(N-\gamma)}\,\mathbf 1_{\{|x|\ge1\}}-A\,e^{-|x|^{2}} .
\]
Then \(\mathbf{w}\) changes sign: indeed \(\mathbf{w}(x)<0\) on \(\{|x|<1\}\),
while \(\mathbf{w}(x)>0\) for all sufficiently large \(|x|\). The substitution
$u=sr^2$ gives
\[
\int_{|x|\ge1}e^{-s|x|^2}|x|^{-(N-\gamma)}\,dx
=
|\sphere^{N-1}|\int_1^\infty e^{-sr^2}r^{\gamma-1}\,dr
=
\frac{|\sphere^{N-1}|}{2}\,s^{-\gamma/2}
\int_s^{\infty}e^{-u}u^{\frac\gamma2-1}\,du,
\]
so that
\[
\int_{|x|\ge1}e^{-s|x|^2}|x|^{-(N-\gamma)}\,dx
\sim
\frac{|\sphere^{N-1}|\,\Gamma(\tfrac\gamma2)}{2}\,s^{-\gamma/2}
\qquad (s\to0^+).
\]
On the other hand,
\[
A\int_{\R^N}e^{-s|x|^2}e^{-|x|^2}\,dx
=
A\Bigl(\frac{\pi}{1+s}\Bigr)^{N/2}
=O(1).
\]
Hence hypothesis \textnormal{(a)} of Theorem~\ref{thm:Tauberblow} holds with
\(\ell\equiv \frac{|\sphere^{N-1}|\,\Gamma(\gamma/2)}{2}\). Since
\(\mathbf{w}_-(x)\le A e^{-|x|^2}\), we have \(\mathbf{w}_-\in L^1(\R^N)\), and
\textnormal{(b)} holds with \(\gamma'=0\). Therefore
Theorem~\ref{thm:Tauberblow} yields
\[
F(R)\sim
\frac{|\sphere^{N-1}|\,\Gamma(\tfrac\gamma2)}
{2\,\Gamma(1+\tfrac\gamma2)}\,R^\gamma
=
\frac{|\sphere^{N-1}|}{\gamma}\,R^\gamma .
\]
This agrees with the direct computation of Example~\ref{ex:power}, since the
Gaussian term contributes only \(O(1)\) to the ball mass. In particular, no
global weak solution exists for \(1<p<p_F(\gamma)\).
\end{example}

\begin{remark}\label{rem:tauberneed}\rm
Some additional hypothesis of the type \textnormal{(b)} is indispensable for
the conclusion \(F\in\RV_\gamma\). Indeed, Tauberian inversion from the
Gaussian-windowed mass to the ball mass requires either monotonicity or a
one-sided Tauberian condition. Here is an explicit example. Let
\(\gamma\in(0,N)\), let \(\Phi_0(r):=r^\gamma(2+\sin r^2)\) for $r\ge0$, and
define the radial function
\[
\mathbf{w}(x):=\frac{\Phi_0'(|x|)}{|\sphere^{N-1}|\,|x|^{N-1}},
\qquad x\neq0 .
\]
Then $|\mathbf{w}(x)|\le C\bigl(|x|^{\gamma-N}+|x|^{\gamma+2-N}\bigr)$, so that
$\mathbf{w}\in L^1_{\mathrm{loc}}(\R^N)$ and $\Lambda_{|\mathbf{w}|}(s)<\infty$ for
every $s>0$. Moreover $F=\Phi_0$. With $U(u):=\Phi_0(\sqrt u)=u^{\gamma/2}(2+\sin u)$,
an integration by parts gives
\[
\Lambda_{\mathbf{w}}(s)=\int_0^\infty e^{-su}\,dU(u)
=s\int_0^\infty e^{-su}u^{\gamma/2}(2+\sin u)\,du
=2\Gamma\bigl(1+\tfrac\gamma2\bigr)s^{-\gamma/2}
+s\,\operatorname{Im}\frac{\Gamma(1+\frac\gamma2)}{(s-i)^{1+\gamma/2}} .
\]
Since $|s-i|\ge1$, the last term is $O(s)$, and (a) holds with
$\ell\equiv2\Gamma(1+\frac\gamma2)$. However, $R^{-\gamma}F(R)=2+\sin R^2$
does not converge, so that $F\notin\RV_\gamma$. Hypothesis (b) fails
here, since $F_-(R)$ grows like $R^{\gamma+2}$. Note that this forcing is
still covered by Theorem~\ref{cor:sign}\,(i), because
$\liminf_{R\to\infty}R^{-\gamma}F(R)=1>0$.
\end{remark}

%==================================================================
\section{Operator regular variation and anisotropic Fujita exponents}
\label{sec:operator}

When the forcing has genuinely different scaling rates along different
directions, the isotropic mass \(F(R)\) may fail to capture the relevant
geometry. A natural substitute is an \emph{anisotropic} mass, measured on
boxes adapted to an exponent matrix \(E\). Throughout this section,
\(E\in\R^{N\times N}\) is symmetric positive definite, with eigenvalues
\[
0<\lambda_{\min}=\lambda_1\le\cdots\le\lambda_N=\lambda_{\max},
\qquad
|E|:=\tr(E)=\sum_{j=1}^N \lambda_j .
\]
Since the Laplacian is invariant under orthogonal changes of variables,
problem \eqref{main} with forcing $\mathbf{w}$ is equivalent to the same
problem with forcing $\mathbf{w}\circ O$ for any orthogonal matrix $O$.
Working in an orthonormal eigenbasis of $E$, we may therefore assume that
\[
E=\diag(\alpha_1,\dots,\alpha_N),
\qquad
\alpha_j=\lambda_j>0.
\]
For \(R>0\), we define the \emph{anisotropic box} and the corresponding
\emph{operator mass} by
\[
\Pi_R:=\Bigl\{x\in\R^N:\ |x_j|\le R^{\alpha_j}\ \ (1\le j\le N)\Bigr\},
\qquad
F^E(R):=\int_{\Pi_R} \mathbf{w}(x)\,dx,
\qquad
|\Pi_R|=2^N R^{|E|}.
\]
The boxes \(\Pi_R\) are comparable, up to fixed dilations, with the
operator balls of Remark~\ref{rem:opolar}, which here read
\(B_R^E=\{x:\sum_jx_j^2R^{-2\alpha_j}\le1\}\). Indeed, $B_R^E\subset\Pi_R\subset B^E_{CR}$ with
$C:=N^{1/(2\lambda_{\min})}\ge1$, since for $x\in\Pi_R$ we have
\[
\sum_{j=1}^Nx_j^2(CR)^{-2\alpha_j}\le\sum_{j=1}^NC^{-2\alpha_j}\le NC^{-2\lambda_{\min}}=1 .
\]
In particular, for \(\mathbf{w}\ge0\), the corresponding masses are
comparable at matching scales:
\[
F_{B^E}(R)\le F^E(R)\le F_{B^E}(CR).
\]
Hence the two notions encode the same anisotropic growth at the level of
two-sided estimates. However, regular variation of one does not follow from
regular variation of the other without additional structure. We work with
boxes, which are more convenient for the test-function estimates.

\begin{theorem}\label{thm:operator}
Let \(E=\diag(\alpha_1,\dots,\alpha_N)\) be positive definite as above, let
\(\mathbf{w}\in L^1_{\mathrm{loc}}(\R^N)\) be nonnegative, and suppose that the
operator mass satisfies
\[
F^E\in\RV_{\gamma_E},\qquad
F^E(R)=R^{\gamma_E}L(R),\quad L\in\RV_0,
\]
for some \(\gamma_E\in(0,|E|)\). Set
\[
q:=|E|-\gamma_E\in(0,|E|),
\qquad
p_F^E:=
\begin{cases}
\dfrac{q}{q-2\lambda_{\min}}, & q>2\lambda_{\min},\\[6pt]
\infty, & q\le 2\lambda_{\min}.
\end{cases}
\]
Then \eqref{main} admits no global weak solution for \(1<p<p_F^E\).
If, in addition, \(q>2\lambda_{\min}\) and \(L(R)\to\infty\), then
nonexistence also holds at the critical exponent \(p=p_F^E\).
\end{theorem}

\begin{proof}
Assume that a global weak solution exists. For $T\ge1$, we set
\(S:=T^{\frac{1}{2\lambda_{\min}}}\ge1\), so that
\[
S^{\alpha_j}=T^{\alpha_j/(2\lambda_{\min})}\ge T^{1/2},
\qquad
\prod_{j=1}^N S^{\alpha_j}=S^{|E|}.
\]
With $f,g$ as in Section~\ref{sec:main}, we define
\[
\psi_T(t,x):=
f(t/T)^{p'}\prod_{j=1}^N\phi_j(x_j),
\qquad
\phi_j(x_j):=g\!\Bigl(\frac{x_j^2}{S^{2\alpha_j}}\Bigr)^{2p'} .
\]
Then \(\psi_T\ge0\), \(\psi_T(T,\cdot)\equiv0\), and
\(\psi_T\equiv1\) on \([T/2,2T/3]\times \Pi_S\). Moreover, \(\psi_T\) is
supported in
\(
[0,T]\times \prod_{j=1}^N[-\sqrt2\,S^{\alpha_j},\sqrt2\,S^{\alpha_j}],
\)
a set of measure \(\lesssim T\,S^{|E|}\).

For the time derivative, exactly as in \eqref{eq:pointwise1}, we have
\(|\partial_t\psi_T|^{p'}\psi_T^{1-p'}\le C\,T^{-p'}\) on the support of
$\psi_T$. For the Laplacian, since each \(\phi_j\) depends only on \(x_j\),
we have
\[
\Delta\psi_T
=
f(t/T)^{p'}\sum_{j=1}^N \phi_j''(x_j)\prod_{i\neq j}\phi_i(x_i).
\]
As in the proof of \eqref{eq:pointwise2}, we obtain
\(|\phi_j''(x_j)|\le C\,S^{-2\alpha_j}\phi_j(x_j)^{1-1/p'}\). Using
$\phi_i\le\phi_i^{1-1/p'}$, we infer that
$|\Delta\psi_T|\le Cf(t/T)^{p'}\bigl(\sum_jS^{-2\alpha_j}\bigr)\bigl(\prod_i\phi_i\bigr)^{1-1/p'}$,
and therefore, on the support of $\psi_T$,
\[
|\Delta\psi_T|^{p'}\psi_T^{1-p'}
\le
C\sum_{j=1}^N S^{-2\alpha_j p'}
\le C\,T^{-p'} ,
\]
because $S^{-2\alpha_j p'}=T^{-\alpha_j p'/\lambda_{\min}}\le T^{-p'}$ for
$T\ge1$, as $\alpha_j\ge\lambda_{\min}$. The Young absorption argument of
Lemma~\ref{lem:twoscale}, with support volume $\le C\,T\,S^{|E|}$, then
gives
\begin{equation}\label{eq:opUB2}
\int_0^T\!\!\!\int_{\R^N}\mathbf{w}\,\psi_T\,dx\,dt
\le C\,T^{1-p'}S^{|E|}.
\end{equation}
On the other hand, since \(\mathbf{w}\ge0\) and \(\psi_T\equiv1\) on
\([T/2,2T/3]\times\Pi_S\), we have
\begin{equation}\label{eq:opLB2}
\int_0^T\!\!\!\int_{\R^N}\mathbf{w}\,\psi_T\,dx\,dt
\ge \frac{T}{6}F^E(S)
=\frac{T}{6}S^{\gamma_E}L(S).
\end{equation}
Combining \eqref{eq:opUB2} and \eqref{eq:opLB2}, and recalling that
\(q= |E|-\gamma_E\) and \(T=S^{2\lambda_{\min}}\), we obtain
\[
L(S)\le C\,T^{-p'}S^{q}
= C\,S^{\,q-2\lambda_{\min}p'}.
\]

If \(q>2\lambda_{\min}\) and \(1<p<p_F^E\), then
\(p'>q/(2\lambda_{\min})\), so that the exponent \(q-2\lambda_{\min}p'\)
is negative. We choose \(\delta\in(0,2\lambda_{\min}p'-q)\). Since
\(L\in\RV_0\), Lemma~\ref{lem:SVbounds} yields \(L(S)\ge S^{-\delta}\) for all
large \(S\), which contradicts the previous estimate. If
\(q\le 2\lambda_{\min}\), then \(q/(2\lambda_{\min})\le1<p'\) for every
\(p>1\), and the same contradiction holds for all \(p>1\). Finally, if
\(q>2\lambda_{\min}\) and \(p=p_F^E\), then \(p'=q/(2\lambda_{\min})\). The
exponent vanishes, and the above estimate reduces to \(L(S)\le C\), which
contradicts the assumption \(L(S)\to\infty\).
\end{proof}

\begin{remark}[Consistency and scaling invariance]\label{rem:opconsist}\rm
For \(E=I\) we have \(\lambda_{\min}=1\) and \(|E|=N\), so that
\(\gamma_E\in(0,N)\) and \(q=N-\gamma_E\). The critical exponent in
Theorem~\ref{thm:operator} then becomes
\[
p_F^E=\frac{N-\gamma_E}{N-\gamma_E-2}=p_F(\gamma_E),
\]
which is consistent with the isotropic exponent. Here \(\Pi_R=[-R,R]^N\)
is a cube, so that the theorem is formulated in terms of cube masses rather
than ball masses. Moreover, the critical exponent is invariant under the
rescaling \(E\mapsto cE\) with \(c>0\). Indeed, the operator mass of
$cE$ at scale $R$ equals $F^E(R^c)$, so that
\[
\lambda_{\min}\mapsto c\lambda_{\min},
\qquad
|E|\mapsto c|E|,
\qquad
\gamma_E\mapsto c\gamma_E,
\qquad
q\mapsto cq,
\]
and the ratio \(\frac{q}{q-2\lambda_{\min}}\) is unchanged. Accordingly, one
may normalise \(E\) conveniently, for instance by imposing
\(\lambda_{\min}=\tfrac12\). As in Corollary~\ref{cor:limsup}, the upper
bound \eqref{eq:opUB2} holds for every $T\ge1$ with a uniform constant, so
that regular variation may be replaced by
$\limsup_{S\to\infty}S^{-\gamma_E}F^E(S)>0$ in the subcritical range.
\end{remark}

\begin{remark}[When is the operator mass regularly varying?]
\label{rem:opORV}\rm
A convenient sufficient condition is the following. Assume that \(\mathbf{w}\)
is operator-regularly varying on $\cone=\R^N\setminus\{0\}$
(Definition~\ref{def:ORV}), with auxiliary function \(V\in\RV_{-q}\) for some
\(q\in(0,|E|)\) and nonnegative profile \(h\not\equiv0\). By
Proposition~\ref{prop:anisohom}, $h$ is \(E\)-homogeneous of degree \(-q\).
Assume moreover that the rescaled family is dominated: there exist
$S_0\ge1$ and $H\in L^1(\Pi_1)$ such that
\[
\frac{\mathbf{w}(S^Ey)}{V(S)}\le H(y)\qquad\text{for a.e. }y\in\Pi_1
\text{ and all }S\ge S_0 .
\]
Writing \(\gamma_E:=|E|-q\in(0,|E|)\), the change of variables \(x=S^Ey\),
for which $dx=S^{|E|}dy$ and $S^E\Pi_1=\Pi_S$, gives
\[
\frac{F^E(S)}{S^{|E|}V(S)}
=\int_{\Pi_1}\frac{\mathbf{w}(S^Ey)}{V(S)}\,dy
\longrightarrow \int_{\Pi_1}h(y)\,dy
\qquad(S\to\infty)
\]
by dominated convergence. The limit is finite, and it is positive because
$h\ge0$ is nontrivial and $E$-homogeneous, hence positive on a set of
positive measure in $\Pi_1$. It follows that
\(F^E(S)\sim \bigl(\int_{\Pi_1}h\bigr)\,S^{|E|}V(S)\), so that
\(F^E\in\RV_{\gamma_E}\). Without a domination assumption, the pointwise
convergence near the origin of $\Pi_1$ is not sufficient to pass to the
limit.

For the diagonal homogeneous weight
\[
\mathbf{w}(x)=\Bigl(\sum_{j=1}^N|x_j|^{1/\alpha_j}\Bigr)^{-q},
\qquad 0<q<|E|,
\]
no limiting argument is needed. We have
$\mathbf{w}(t^Ex)=t^{-q}\mathbf{w}(x)$, and the same change of variables gives
\(F^E(S)=S^{|E|-q}F^E(1)=S^{\gamma_E}F^E(1)\) exactly, where
\(F^E(1)<\infty\) because $q<|E|$. Thus \(L\equiv\mathrm{const}\), and
Theorem~\ref{thm:operator} yields nonexistence for
\(1<p<\frac{q}{q-2\lambda_{\min}}\), and for all \(p>1\) if
\(q\le 2\lambda_{\min}\).
\end{remark}

The next two examples illustrate the relation between the operator
threshold and the isotropic one; see also Figure~\ref{fig:operator}. For
``filled'' homogeneous weights, the isotropic criterion typically yields the
larger nonexistence range. For weights concentrated along a thin operator
orbit, the operator criterion can be strictly sharper.

\begin{example}[A filled weight: the isotropic criterion dominates]\label{ex:R3}
Let $N=3$ and $E=\diag(\tfrac12,1,1)$, so that $|E|=\tfrac52$ and
$\lambda_{\min}=\tfrac12$. For $0<q<\tfrac52$, let
\[
\mathbf{w}(x)=Q(x)^{-q}\,\mathbf 1_{\{Q(x)\ge1\}},
\qquad Q(x):=x_1^2+|x_2|+|x_3| .
\]
Since $Q(t^Ex)=tQ(x)$, the function $\mathbf{w}$ is $E$-homogeneous of degree
$-q$ outside the bounded set $\{Q<1\}$. The change of variables of
Remark~\ref{rem:opORV} gives
$F^E(R)=R^{5/2-q}\int_{\Pi_1\cap\{Q\ge1/R\}}Q^{-q}\,dy\sim cR^{5/2-q}$, so that
the operator theorem yields nonexistence for
\[
1<p<\frac{q}{q-1}\quad(q>1),\qquad\text{and for all }p>1\ \text{ if } q\le1.
\]
On the other hand, splitting the ball $\{|x|\le R\}$ according to the dyadic
level sets of $Q$, one checks that the Euclidean mass satisfies
\[
F(R)\asymp
\begin{cases}
R^{3-2q}, & 0<q<\tfrac12,\\[1mm]
R^{2}\log R, & q=\tfrac12,\\[1mm]
R^{5/2-q}, & \tfrac12< q<\tfrac52 .
\end{cases}
\]
Therefore Theorem~\ref{thm:RVblowup}\,(i)\,(a) gives nonexistence for all
\(p>1\) when \(q\le\tfrac32\), and for
\[
1<p<\frac{q+1/2}{q-3/2}
\qquad\text{when } q\in(\tfrac32,\tfrac52).
\]
Since $(q+\frac12)(q-1)-q(q-\frac32)=q-\frac12>0$, the isotropic threshold is
at least as large as the operator one for every $q$, and strictly larger for
$q\in(1,\tfrac52)$. The reason is geometric: the Euclidean ball of radius $R$
already captures the same mass order $R^{5/2-q}$ as the operator box, so
anisotropy brings nothing here.
\end{example}

\begin{example}[A thin weight: the operator criterion can win]\label{ex:N4}
Let $N=4$, write $x=(x',x_4)\in\R^3\times\R$, and let
$E=\diag(\tfrac12,\tfrac12,\tfrac12,1)$, so that $|E|=\tfrac52$ and
$\lambda_{\min}=\tfrac12$. For $q\in(0,\tfrac52)$ we put
\[
\mathbf{w}(x)=x_4^{-q}\,\mathbf 1_{\Omega}(x),
\qquad
\Omega:=\bigl\{x_4\ge1,\ |x'|\le x_4^{1/2}\bigr\}.
\]
The paraboloidal region $\{x_4>0,\ |x'|\le x_4^{1/2}\}$ is invariant under
$t^Ex=(t^{1/2}x',t\,x_4)$, and $\mathbf{w}(t^Ex)=t^{-q}\mathbf{w}(x)$ whenever $x$
and $t^Ex$ lie in $\Omega$. Thus $\mathbf{w}$ is $E$-homogeneous of degree
$-q$ away from the bounded region $\{x_4<1\}$. The operator box $\Pi_R$
meets $\Omega$ exactly in $\{1\le x_4\le R,\ |x'|\le x_4^{1/2}\}$, because
$|x'|\le x_4^{1/2}\le R^{1/2}$ there. Hence
\[
F^E(R)=\int_1^R x_4^{-q}\,\omega_3\,x_4^{3/2}\,dx_4
=\frac{\omega_3}{\tfrac52-q}\bigl(R^{\,5/2-q}-1\bigr),
\]
so that $F^E\in\RV_{5/2-q}$ with $L\to\mathrm{const}$. Theorem~\ref{thm:operator}
(with $2\lambda_{\min}=1$) gives nonexistence for
\[
1<p<\frac{q}{q-1}\quad(q>1),\qquad\text{and for all }p>1\ \text{ if } q\le1.
\]
The crucial point is that the \emph{Euclidean} mass has the same index.
Indeed, since $\Omega$ is thin, the ball $\{|x|\le R\}$ meets $\Omega$ in a
set containing $\{1\le x_4\le R/2,\ |x'|\le x_4^{1/2}\}$ and contained in
$\{1\le x_4\le R,\ |x'|\le x_4^{1/2}\}$, and we obtain
$F(R)\asymp R^{5/2-q}$. With $\gamma=\tfrac52-q$ we have
$N-\gamma=q+\tfrac32$, and the isotropic theorem yields only
\[
1<p<p_F(\gamma)=\frac{q+3/2}{q-1/2}\quad(q>\tfrac12),
\qquad\text{and all }p>1\ \text{ if } q\le\tfrac12.
\]
For $q\in(\tfrac12,1]$, the operator theorem gives \emph{all} $p>1$, while
the isotropic one does not. For $q\in(1,\tfrac32)$, we have
$q(q-\frac12)-(q+\frac32)(q-1)=\frac32-q>0$, so that
$\frac{q}{q-1}>\frac{q+3/2}{q-1/2}$. Hence the operator threshold is
\emph{strictly larger} for every $q\in(\tfrac12,\tfrac32)$. For instance, at
$q=\tfrac54$ the operator theorem gives nonexistence for $p<5$, whereas the
isotropic theorem only gives $p<\tfrac{11}{3}\approx3.67$. The two
thresholds coincide at $q=\frac32$, and the isotropic one is larger for
$q\in(\frac32,\frac52)$.
\end{example}

\begin{figure}[H]
\centering
\begin{tikzpicture}
\begin{groupplot}[
  group style={group size=2 by 1, horizontal sep=1.6cm},
  width=7.4cm, height=6cm,
  xmin=0, xmax=2.5, ymin=1, ymax=10,
  xlabel={$q$},
  axis lines=left,
  grid=major, grid style={figGrid, very thin, opacity=0.5},
  tick label style={font=\small}, label style={font=\small},
  title style={font=\small},
  legend style={font=\scriptsize, draw=none, fill=none, at={(0.98,0.98)}, anchor=north east},
  legend cell align=left,
]
\nextgroupplot[title={Example~\ref{ex:R3} ($N=3$)}, ylabel={nonexistence threshold}]
\addplot[figA, line width=1.2pt, domain=1.112:2.5, samples=150] {x/(x-1)};
\addlegendentry{operator $\frac{q}{q-1}$}
\addplot[figB, line width=1.2pt, dashed, domain=1.723:2.5, samples=150] {(x+0.5)/(x-1.5)};
\addlegendentry{isotropic $\frac{q+1/2}{q-3/2}$}
\nextgroupplot[title={Example~\ref{ex:N4} ($N=4$)}]
\addplot[figA, line width=1.2pt, domain=1.112:2.5, samples=150] {x/(x-1)};
\addlegendentry{operator $\frac{q}{q-1}$}
\addplot[figB, line width=1.2pt, dashed, domain=0.723:2.5, samples=150] {(x+1.5)/(x-0.5)};
\addlegendentry{isotropic $\frac{q+3/2}{q-1/2}$}
\addplot[only marks, mark=*, mark size=1.6pt, black] coordinates {(1.5,3)};
\end{groupplot}
\end{tikzpicture}
\caption{Nonexistence thresholds given by the operator theorem
(Theorem~\ref{thm:operator}, solid) and by the isotropic theorem
(Theorem~\ref{thm:RVblowup}, dashed) as functions of $q$. Each curve is
infinite to the left of its vertical asymptote, which means that
nonexistence holds for every $p>1$ there. Left: for the filled weight of
Example~\ref{ex:R3}, the isotropic threshold is always at least as large.
Right: for the thin weight of Example~\ref{ex:N4}, the operator threshold is
strictly larger for $q\in(\frac12,\frac32)$; the curves cross at $q=\frac32$
(dot).}
\label{fig:operator}
\end{figure}

\begin{remark}[Strictly anisotropic limit]\label{rem:anisolimit}\rm
Theorem~\ref{thm:anisotropic} is the formal degenerate limit of
Theorem~\ref{thm:operator} in which $k$ eigenvalues tend to $0$, the
corresponding directions being no longer rescaled. Writing
$E_\eta=\diag(\eta,\dots,\eta,1,\dots,1)$ with $k$ entries $\eta\to0^+$, the
operator box degenerates to the slab on which $\widetilde{\mathbf{w}}$ is
integrated, and the index $|E_\eta|-q$ tends to $N-k$. The rigorous,
non-degenerate statement is Theorem~\ref{thm:anisotropic}; we record the
limit only as a heuristic link.
\end{remark}

\begin{remark}\label{rem:opsharp}\rm
The exponent \(p_F^E\) is \emph{not} claimed to be sharp. The test-function
method applies more generally to rectangular cut-offs with side lengths
\(T^{a_j}\), provided \(a_j\ge \tfrac12\) for every \(j\). Optimising the
resulting contradiction over all such anisotropic scalings, rather than
restricting to the operator orbit \(a_j=\alpha_j/(2\lambda_{\min})\), may
yield a larger nonexistence range; see Section~\ref{sec:open}.

The assumption that \(E\) be symmetric positive definite is imposed mainly
to retain the clean operator-polar framework of Remark~\ref{rem:opolar}.
After orthogonal diagonalisation, the proof above only uses diagonal
matrices with positive entries. It is therefore natural to ask whether
analogous results remain valid for more general expansive matrices,
including non-diagonalisable ones. In that setting one expects additional
features, such as Jordan blocks and logarithmic corrections, and we leave
this question open.
\end{remark}

%==================================================================
\section{Nonlocal and fractional problems: the $\Phi$-formalism}
\label{sec:nonlocal}
In this section, we extend our study in two directions: nonlocal forcing
terms of Riesz-potential type, and diffusion driven by the fractional
Laplacian.

\subsection{Riesz-potential forcings}

Let $0<\sigma<N$, and consider a forcing of Riesz-potential type,
\[
\mathbf{w}(x,t):=\bigl(I_\sigma g(\cdot,t)\bigr)(x)
=c_{N,\sigma}\int_{\R^N}\frac{g(y,t)}{|x-y|^{N-\sigma}}\,dy,
\qquad g\ge0,
\]
where $I_\sigma$ is the Riesz potential. We assume that $g$ is measurable
and that $\mathbf{w}\in L^1_{\mathrm{loc}}([0,\infty)\times\R^N)$. The relevant
lower bound is on a \emph{truncated} mass of $g$.

\begin{theorem}\label{thm:riesz}
Let $0<\sigma<N$, $m>-1$, and suppose that $g\ge0$ satisfies, for some
$c_0>0$ and $R_0\ge1$,
\begin{equation}\label{eq:rieszmass}
\int_{|y|\le R}g(y,t)\,dy\ \ge\ c_0\,t^{m}
\qquad\text{for all } t\ge1,\ R\ge R_0 .
\end{equation}
Then $\mathbf{w}:=I_\sigma g$ satisfies
\[
\int_{|x|\le R}\mathbf{w}(x,t)\,dx\ \ge\ c\,R^{\sigma}\,t^{m}
\qquad\text{for all } t\ge1,\ R\ge R_0,
\]
and consequently \eqref{eq:mainST} admits no global weak solution for
\[
1<p<p_F(\sigma,m)=\frac{N-\sigma-2m}{\,N-\sigma-2m-2\,}
\]
(for all \(p>1\) if \(N-\sigma-2m\le2\)), and for every $p>1$ if $m>0$.
\end{theorem}

\begin{proof}
Let $t\ge1$ and $R\ge R_0$. By Tonelli's theorem,
\[
\int_{|x|\le R}\mathbf{w}(x,t)\,dx
=c_{N,\sigma}\int_{\R^N}g(y,t)
\Bigl(\int_{|x|\le R}\frac{dx}{|x-y|^{N-\sigma}}\Bigr)dy .
\]
We restrict the outer integral to \(|y|\le R\). If \(|y|\le R\) and
\(|x|\le R/2\), then \(|x-y|\le\frac32R\). Hence
\[
\int_{|x|\le R}\frac{dx}{|x-y|^{N-\sigma}}
\ge
\int_{|x|\le R/2}\frac{dx}{|x-y|^{N-\sigma}}
\ge
\left(\tfrac32 R\right)^{-(N-\sigma)}\,|B(0,R/2)|
= c\,R^{\sigma},
\]
where \(c>0\) depends only on \(N\) and \(\sigma\). Therefore, using
\eqref{eq:rieszmass}, we obtain
\[
\int_{|x|\le R}\mathbf{w}(x,t)\,dx
\ge c_{N,\sigma}c\,R^{\sigma}\int_{|y|\le R}g(y,t)\,dy
\ge c'\,R^{\sigma}t^m .
\]
This is precisely hypothesis \eqref{eq:STmass} with \(\gamma=\sigma\), and
Theorem~\ref{thm:spacetime} applies.
\end{proof}

\begin{remark}\label{rem:riesztrunc}
\rm Since $R\mapsto\int_{|y|\le R}g(y,t)\,dy$ is nondecreasing, condition
\eqref{eq:rieszmass} simply says that the mass of $g(\cdot,t)$ in the
\emph{fixed} ball $\{|y|\le R_0\}$ is at least $c_0t^m$. This truncation is
essential: a global lower bound of the form
\(\int_{\R^N} g(y,t)\,dy \ge c_0 t^m\) does \emph{not} suffice by itself.
Indeed, a unit mass located at a point $y$ with \(|y|\gg R\) contributes
only
\[
\int_{|x|\le R}\frac{dx}{|x-y|^{N-\sigma}}
\asymp R^{N}|y|^{-(N-\sigma)}
\]
to the mass of $\mathbf{w}$ in the ball of radius $R$, and this tends to zero
as \(|y|\to\infty\). Thus a lower bound on the total mass alone does not
prevent the source from escaping to spatial infinity, whereas
\eqref{eq:rieszmass} rules out this loss.
\end{remark}

\subsection{Fractional diffusion}

We replace the Laplacian by the fractional Laplacian $(-\Delta)^s$,
$s\in(0,1)$:
\begin{equation}\label{eq:mainfrac}
\partial_t u+(-\Delta)^s u=|u|^p+\mathbf{w}(x,t),
\qquad (t,x)\in(0,\infty)\times\R^N,\qquad u(0,\cdot)=u_0\ge0 .
\end{equation}
For $v\in C^2_c(\R^N)$ we use the pointwise definition
\[
(-\Delta)^sv(x)=\frac{c_{N,s}}2\int_{\R^N}\frac{2v(x)-v(x+z)-v(x-z)}{|z|^{N+2s}}\,dz ,
\]
which is absolutely convergent. Moreover, if $v$ is supported in
$\{|x|\le\rho\}$, then $|(-\Delta)^sv(x)|\le C(1+|x|)^{-(N+2s)}$, with $C$
depending on $\rho$ and on $\|v\|_{C^2}$. Indeed, for $|x|\ge2\rho$ we have
$(-\Delta)^sv(x)=-c_{N,s}\int_{|y|\le\rho}v(y)|x-y|^{-N-2s}\,dy$ and
$|x-y|\ge|x|/2$.

A \emph{global weak solution} of \eqref{eq:mainfrac} is a nonnegative
function $u\in L^p_{\mathrm{loc}}([0,\infty)\times\R^N)$ such that
$u\,(1+|x|)^{-(N+2s)}\in L^1((0,T)\times\R^N)$ for every $T>0$, and
\[
\int_0^T\!\!\!\int_{\R^N}(|u|^p+\mathbf{w})\varphi\,dx\,dt
+\int_{\R^N}u_0\varphi(0,\cdot)\,dx
=-\int_0^T\!\!\!\int_{\R^N}u\,\partial_t\varphi\,dx\,dt
+\int_0^T\!\!\!\int_{\R^N}u\,(-\Delta)^s\varphi\,dx\,dt
\]
for every $T>0$ and every nonnegative $\varphi\in C^{1}([0,T];C^2_c(\R^N))$
supported in $[0,T]\times\{|x|\le\rho\}$ for some $\rho$, with
$\varphi(T,\cdot)\equiv0$. All terms are finite by the decay estimate above.
The sign in front of the last integral follows from the symmetry of
$(-\Delta)^s$; for $s=1$ it reduces to \eqref{eq:weakform}. The nonlocal
test-function computation is made rigorous by the convexity inequality of
C\'ordoba and C\'ordoba \cite{CordobaCordoba}; see also \cite{Ju2005}.

\begin{lemma}[C\'ordoba--C\'ordoba inequality]\label{lem:cordoba}
Let $s\in(0,1)$, let $\Theta\in C^1(\R)$ be convex, and let
$\Phi\in C^2_c(\R^N)$ be such that $\Theta\circ\Phi-\Theta(0)\in C^2_c(\R^N)$.
Then, pointwise,
\[
(-\Delta)^s\bigl(\Theta\circ\Phi\bigr)(x)
\ \le\ \Theta'\bigl(\Phi(x)\bigr)\,(-\Delta)^s\Phi(x),
\qquad x\in\R^N .
\]
\end{lemma}

\begin{proof}
By convexity, $\Theta(b)\ge\Theta(a)+\Theta'(a)(b-a)$ for all $a,b\in\R$.
Applying this with $a=\Phi(x)$ and $b=\Phi(x\pm z)$ and adding the two
inequalities, we obtain
\[
2\Theta(\Phi(x))-\Theta(\Phi(x+z))-\Theta(\Phi(x-z))
\le\Theta'(\Phi(x))\bigl(2\Phi(x)-\Phi(x+z)-\Phi(x-z)\bigr).
\]
Integrating against the positive kernel $\frac{c_{N,s}}2|z|^{-N-2s}$, which is
legitimate since both sides are absolutely integrable, we obtain the claim.
\end{proof}

We apply Lemma~\ref{lem:cordoba} with $\Theta(r)=|r|^{2p'}$, which is
convex and of class $C^2$ since $2p'>2$, and with a nonnegative bump
$\Phi\in C_c^\infty(\R^N)$, so that $\Theta\circ\Phi=\Phi^{2p'}\in C^2_c(\R^N)$.

\begin{theorem}\label{thm:frac}
Let \(s\in(0,1)\), \(N\ge1\), \(m>-1\), and \(\gamma\ge0\). Assume that
\(\mathbf{w}\ge0\) satisfies
\[
\int_{|x|\le R} \mathbf{w}(x,t)\,dx \ge c_0\,t^m R^\gamma
\qquad\text{for all } t\ge1,\ R\ge1,
\]
for some constant \(c_0>0\). Then \eqref{eq:mainfrac} admits no global weak
solution for
\[
1<p<p_F^{\,s}(\gamma,m):=
\frac{N-\gamma-2sm}{\,N-\gamma-2sm-2s\,},
\]
with the convention that \(p_F^{\,s}(\gamma,m)=\infty\) whenever
\(N-\gamma-2sm\le 2s\). Moreover, if \(m>0\), then \eqref{eq:mainfrac}
admits no global weak solution for every \(p>1\).
\end{theorem}

\begin{proof}
Assume, for contradiction, that \eqref{eq:mainfrac} admits a global weak
solution. We choose \(\Phi\in C_c^\infty(\R^N)\) such that
\[
0\le \Phi\le 1,\qquad
\Phi\equiv1 \ \text{on } \{|x|\le1\},\qquad
\operatorname{supp} \Phi\subset \{|x|\le2\}.
\]
For \(T,R>0\), we set \(\Phi_R(x):=\Phi(x/R)\), \(f_T(t):=f(t/T)\), where
\(f\) is the temporal cut-off of Section~\ref{sec:main}, and
\[
\psi_{T,R}(t,x):=f_T(t)^{p'}\,\Phi_R(x)^{2p'} .
\]
(Note that, in this proof, $f_T$ denotes $f(t/T)$ without the power $p'$.)
By scaling, $(-\Delta)^s\Phi_R(x)=R^{-2s}\bigl((-\Delta)^s\Phi\bigr)(x/R)$.
Since $|(-\Delta)^s\Phi(x)|\le C(1+|x|)^{-(N+2s)}$ and $(N+2s)p'>N$, we
obtain
\[
\int_{\R^N}\bigl|(-\Delta)^s\Phi_R\bigr|^{p'}\,dx
=
R^{N-2sp'}\int_{\R^N}\bigl|(-\Delta)^s\Phi\bigr|^{p'}\,dx
=
C\,R^{N-2sp'}.
\]

We now test \eqref{eq:mainfrac} against \(\psi:=\psi_{T,R}\). The initial
term vanishes since $f(0)=0$. Moreover, since \(u\ge0\),
Lemma~\ref{lem:cordoba} yields
\[
(-\Delta)^s(\Phi_R^{2p'})
\le 2p'\,\Phi_R^{2p'-1}(-\Delta)^s\Phi_R
\le 2p'\,\Phi_R^{2p'-1}\bigl|(-\Delta)^s\Phi_R\bigr|.
\]
Therefore, we obtain
\[
\int_0^T\!\!\!\int_{\R^N}|u|^p\psi\,dx\,dt
+\int_0^T\!\!\!\int_{\R^N}\mathbf{w}\,\psi\,dx\,dt
\le
2p'\int_0^T\!\!\!\int_{\R^N}
u\,f_T^{p'}\,\Phi_R^{2p'-1}\bigl|(-\Delta)^s\Phi_R\bigr|
+\int_0^T\!\!\!\int_{\R^N}
u\,\bigl|\partial_t(f_T^{p'})\bigr|\,\Phi_R^{2p'} .
\]
Observe that \(\psi^{1/p}=f_T^{p'-1}\Phi_R^{2p'-2}\), since
\(\frac{p'}{p}=p'-1\). Thus
\[
u\,f_T^{p'}\Phi_R^{2p'-1}\bigl|(-\Delta)^s\Phi_R\bigr|
=
\bigl(u\psi^{1/p}\bigr)\,\bigl(f_T\Phi_R|(-\Delta)^s\Phi_R|\bigr),
\qquad
u\,|\partial_t(f_T^{p'})|\,\Phi_R^{2p'}
=
\bigl(u\psi^{1/p}\bigr)\,
\bigl(p'|f_T'|\,\Phi_R^{2}\bigr).
\]
We apply Young's inequality $ab\le\frac14a^p+Cb^{p'}$ to both terms and
absorb \(\frac12\int\!\!\int |u|^p\psi\), which is finite. Using
$0\le f_T,\Phi_R\le1$, we arrive at
\[
\frac12\int_0^T\!\!\!\int_{\R^N}|u|^p\psi\,dx\,dt
+\int_0^T\!\!\!\int_{\R^N}\mathbf{w}\,\psi\,dx\,dt
\le
C\,T\int_{\R^N}\bigl|(-\Delta)^s\Phi_R\bigr|^{p'}\,dx
+C\int_0^T |f_T'|^{p'}\,dt
\int_{\R^N}\Phi_R^{2p'}\,dx.
\]
Since \(\int_0^T |f_T'|^{p'}\,dt = T^{1-p'}\int_0^1|f'|^{p'}\) and
\(\int_{\R^N}\Phi_R^{2p'}\,dx \le C R^N\), we deduce the fractional two-scale
bound
\begin{equation}\label{eq:fractwoscale}
\int_0^T\!\!\!\int_{\R^N}\mathbf{w}\,\psi_{T,R}\,dx\,dt
\le C\bigl(T\,R^{N-2sp'} + T^{1-p'}R^N\bigr),
\qquad T,R>0 .
\end{equation}

For the lower bound, note that \(\psi_{T,R}\equiv1\) on
$[T/2,\,2T/3]\times \{|x|\le R\}$. For $T\ge2$ and $R\ge1$, the hypothesis on
\(\mathbf{w}\) gives
\[
\int_0^T\!\!\!\int_{\R^N}\mathbf{w}\,\psi_{T,R}\,dx\,dt
\ge
c_0R^\gamma\int_{T/2}^{2T/3} t^m\,dt
\ge c\,T^{\,1+m}R^\gamma .
\]

We first choose the parabolic scale \(R=T^{1/(2s)}\). Both terms on the
right-hand side of \eqref{eq:fractwoscale} then equal
$T^{\,1-p'+\frac{N}{2s}}$, and we obtain
\[
T^{\,1+m+\frac{\gamma}{2s}}
\lesssim
T^{\,1-p'+\frac{N}{2s}} .
\]
This is impossible for large \(T\) whenever
$1+m+\frac{\gamma}{2s}>1-p'+\frac{N}{2s}$, that is,
$p'>\frac{N-\gamma-2sm}{2s}$. If \(N-\gamma-2sm>2s\), this is equivalent to
\(1<p<p_F^{\,s}(\gamma,m)\). If \(N-\gamma-2sm\le 2s\), then
\(\frac{N-\gamma-2sm}{2s}\le1<p'\) for every \(p>1\), and the same
contradiction holds for all \(p>1\).

Finally, let $m>0$. We freeze the spatial scale, $R=1$, and obtain
$c\,T^{1+m}\le C(T+T^{1-p'})$ for $T\ge2$. This is impossible for large $T$,
whatever $p>1$. This completes the proof.
\end{proof}

\begin{remark}\label{rem:fraclowdim}\rm
If \(N-\gamma-2sm\le2s\), then \(p_F^{\,s}(\gamma,m)=\infty\), and
Theorem~\ref{thm:frac} yields nonexistence for every \(p>1\). In
particular, this covers all \(\gamma\ge0\) whenever \(N\le 2s(1+m)\). For
\(\gamma=0\) and \(m=0\), the threshold is \(\frac{N}{N-2s}\), in agreement
with the integrable case \cite{Majdoub2023,KumarTorebek}. Formally, for
\(s=1\) and \(m=0\) the exponent reduces to \eqref{eq:pFgamma}.
\end{remark}

\subsection{The cumulative-forcing functional}

The space--time and fractional results are governed by a single regularly
varying object. We fix the diffusion order $s\in(0,1]$, with $s=1$ the
classical case, and the natural parabolic radius $R(T):=T^{1/(2s)}$. We
define the \emph{cumulative-forcing functional}
\[
\Phi(T):=\int_0^T\!\!\int_{|x|\le R(T)}\mathbf{w}(x,t)\,dx\,dt .
\]

\begin{proposition}[$\Phi$ is regularly varying]\label{prop:PhiRV}
Suppose that there exist $m>-1$, $\gamma\ge0$ and $L\in\RV_0$ such that
\begin{equation}\label{eq:PhiHyp}
\int_{|x|\le R}\mathbf{w}(x,t)\,dx
= t^{m}R^{\gamma}L(R)\bigl(1+\eps(R,t)\bigr),
\qquad
\sup_{t\ge1}|\eps(R,t)|\xrightarrow[R\to\infty]{}0,
\end{equation}
and that $\int_0^1\!\int_{|x|\le R}|\mathbf{w}|\,dx\,dt=O\bigl(R^{\gamma}L(R)\bigr)$.
Then
\[
\Phi(T)\ \sim\ \frac{1}{m+1}\,T^{\,m+1}\,R(T)^{\gamma}L\bigl(R(T)\bigr)
=\frac{1}{m+1}\,T^{\,m+1+\frac{\gamma}{2s}}\,L\bigl(T^{1/(2s)}\bigr),
\]
so that $\Phi\in\RV_{\alpha}$ with index
\[
\alpha:=m+1+\frac{\gamma}{2s}.
\]
\end{proposition}

\begin{proof}
We split $\int_0^T=\int_0^1+\int_1^T$. The first part is
$O(R(T)^\gamma L(R(T)))=o(T^{m+1}R(T)^\gamma L(R(T)))$, since $m+1>0$. In
the second part, \eqref{eq:PhiHyp} with $R=R(T)$ shows that the integrand
equals $t^m R(T)^\gamma L(R(T))(1+o(1))$, uniformly in $t\in[1,T]$.
Moreover, $\int_1^T t^m\,dt\sim T^{m+1}/(m+1)$. The claim follows, since
$T\mapsto L(T^{1/(2s)})$ is clearly slowly varying.
\end{proof}

\begin{theorem}[$\Phi$-dichotomy]\label{thm:Phidich}
Let $\mathbf{w}\ge0$ satisfy \eqref{eq:PhiHyp}, and set
$\beta(p):=1+\frac{N}{2s}-p'$ and $\alpha=m+1+\frac{\gamma}{2s}$.
\begin{enumerate}[label=\textnormal{(\roman*)},leftmargin=*]
\item If $\alpha>\beta(p)$, then \eqref{eq:mainfrac} (or \eqref{eq:mainST}
for $s=1$) admits no global weak solution.
\item If $\alpha=\beta(p)$ and $L(R)\to\infty$, then again no global weak
solution exists.
\end{enumerate}
The borderline identity \(\alpha=\beta(p)\) is equivalent to
\(p=p_F^{\,s}(\gamma,m)\).
\end{theorem}

\begin{proof}
Let \(\psi_T\) be the space--time test function at the parabolic scale
$R(T)=T^{1/(2s)}$, namely $\psi_{T,R(T)}$ from the proof of
Theorem~\ref{thm:frac} if $s<1$, and from Lemma~\ref{lem:twoscale} if $s=1$.
By \eqref{eq:fractwoscale} (respectively \eqref{eq:twoscale}), we have the
upper bound
\[
\int_0^T\!\!\!\int_{\R^N}\mathbf{w}\,\psi_T\,dx\,dt
\le C\,T^{\beta(p)},
\qquad
\beta(p)=1+\frac{N}{2s}-p'.
\]
On the other hand, \(\psi_T\equiv1\) on
$[T/2,2T/3]\times \{|x|\le R(T)\}$ and $\mathbf{w}\ge0$. Using
\eqref{eq:PhiHyp} with \(R=R(T)\), we obtain, for large $T$,
\begin{align*}
\int_0^T\!\!\!\int_{\R^N}\mathbf{w}\,\psi_T\,dx\,dt
&\ge
\int_{T/2}^{2T/3}\int_{|x|\le R(T)} \mathbf{w}(x,t)\,dx\,dt\\
&=
R(T)^\gamma L(R(T))(1+o(1))
\int_{T/2}^{2T/3} t^m\,dt
\ge
c\,T^\alpha L(R(T)) .
\end{align*}
Combining the upper and lower bounds, we arrive at
\(T^\alpha L(R(T)) \lesssim T^{\beta(p)}\). If \(\alpha>\beta(p)\), the
powers of \(T\) already contradict each other as \(T\to\infty\), by
Lemma~\ref{lem:SVbounds}. If \(\alpha=\beta(p)\), the inequality reduces to
$L(R(T))\lesssim 1$, which is impossible if \(L(R)\to\infty\). This proves
both assertions. Finally, the relation \(\alpha=\beta(p)\) reads
$m+1+\frac{\gamma}{2s}=1+\frac{N}{2s}-p'$, that is,
$p'=\frac{N-\gamma-2sm}{2s}$, which is equivalent to
\(p=p_F^{\,s}(\gamma,m)\).
\end{proof}

\begin{remark}\label{rem:Phisilent}
\rm When $\alpha<\beta(p)$, that is, $p>p_F^{\,s}(\gamma,m)$, the two
bounds at the parabolic scale are consistent, and the test-function method
at this scale yields no obstruction. The meaning of this regime depends on
the sign of $m$. For $m>0$, it is \emph{not} a global existence regime:
freezing the spatial scale instead of using $R(T)$ gives nonexistence for
every $p>1$ (Theorems~\ref{thm:spacetime} and~\ref{thm:frac}). For
$-1<m\le0$, the parabolic scale is optimal, and $\alpha<\beta(p)$ is the
regime in which global solutions for small data are expected. For
$\gamma=0$, this is known in the local case: global solutions exist for
small data when $p>\frac{N}{N-2}$ and $m=0$ \cite{BandleLevineZhang2000},
and when $p>\frac{N-2m}{N-2m-2}$ and $-1<m<0$ \cite{JleliKawakamiSamet2020}.
For $\gamma>0$, this is an open question. Thus, for $-1<m\le0$, the
$\Phi$-formalism reduces the nonexistence theorems of
Sections~\ref{sec:ext} and~\ref{sec:nonlocal} to the single scalar
comparison $\alpha\gtrless\beta(p)$.
\end{remark}

%==================================================================
\section{Future research directions}\label{sec:open}

We conclude with several questions that arise naturally from the present
study. Some of them seem accessible with the tools developed here, while
others pose genuinely new challenges.

\begin{enumerate}[label=\textnormal{(\arabic*)},leftmargin=*]
\item \emph{Critical case for non-amplifying masses.} Let $F\in\RV_\gamma$
with $\gamma\in[0,N-2)$ and a slowly varying factor $L$ which is
\emph{bounded} or \emph{decaying} (Example~\ref{ex:decayL}). Then the
test-function method is silent at $p=p_F(\gamma)$. For $\gamma=0$ and
$L(R)\to\int\mathbf{w}>0$, it is known that the critical exponent
$\frac{N}{N-2}$ still belongs to the nonexistence regime, under mild
additional assumptions \cite{BandleLevineZhang2000}. For $\gamma>0$, the
situation is unclear. For power forcings
$\mathbf{w}\sim\eps|x|^{-(N-\gamma)}$, the critical case corresponds to
$N-\gamma=2p'$, and the existence of positive singular stationary
supersolutions of the form $A|x|^{-2/(p-1)}$ for small $\eps$ suggests that
small global solutions might exist, in contrast with $\gamma=0$; compare
with the biharmonic results of \cite{TobakhanovTorebek}. More generally, can
the critical behaviour be classified in terms of the de~Haan class of $L$?

\item \emph{Sufficiency of the universal mass bound.} By
Corollary~\ref{cor:limsup}, global solvability forces the bound
$F(R)\le K_{N,p}R^{N-2p'}$ for all $R>0$. For nonnegative forcings and
$p>\frac{N}{N-2}$, is the condition
\[
\sup_{R>0}R^{2p'-N}F(R)<\infty,
\]
possibly with a small constant, also sufficient for the existence of global
solutions with small initial data? The sparse forcings of
Proposition~\ref{prop:limsupweak}, whose mass oscillates between two
regimes, are natural test cases.

\item \emph{Sharpness of the operator exponent.} Is $p_F^E$ in
Theorem~\ref{thm:operator} optimal? More precisely, one may optimise the
test-function threshold over \emph{all} rectangular scalings
$(T^{a_j})_{j}$ with $a_j\ge\frac12$, not only over the operator orbit
$a_j=\alpha_j/(2\lambda_{\min})$, and compare the result with the isotropic
bound. A complete answer would also require matching global existence
results under the operator regular variation hypothesis.

\item \emph{Non-diagonalisable exponents.} Extend
Theorem~\ref{thm:operator} to exponent matrices $E$ with nontrivial Jordan
structure or complex spectrum, where logarithmic corrections to the polar
calculus appear (cf.\ \cite{MS}).

\item \emph{Two-sided Tauberian theory.} Remark~\ref{rem:tauberneed} shows
that the one-sided control (b) in Theorem~\ref{thm:Tauberblow} cannot be
simply dropped. Can it be weakened, for instance to a slow-oscillation
condition on $F$? Characterise the sign-changing forcings whose net mass is
regularly varying purely in terms of $\Lambda_{\mathbf{w}}$.

\item \emph{Systems and other diffusions.} Develop the regularly varying
Fujita theory for reaction--diffusion systems and for mixed
local--nonlocal operators $\mathscr L=-\Delta+(-\Delta)^s$, where two
scaling regimes compete \cite{KumarTorebek,arXiv}.

\item \emph{Sharp lifespan asymptotics.} Beyond nonexistence, obtain
upper and lower bounds on the lifespan $T_{\max}(u_0)$ in the subcritical
range, with the regularly varying factor $L$ entering the rate.
\end{enumerate}

%==================================================================
%\section*{Acknowledgements}
%The authors thank the anonymous referees for their careful reading.

\section*{Declarations}
\noindent\textbf{Competing interests.} The authors declare no competing
interests.\\
\textbf{Data availability.} Data sharing is not applicable: no datasets
were generated or analysed.

\section*{Acknowledgement}

This work was completed while VK was visiting the Ghent Analysis \& PDE
Center at Ghent University. He gratefully acknowledges the financial support
and the excellent research facilities provided by the center.

%==================================================================

\end{document}